\documentclass[12pt]{article}
\usepackage{amsmath,enumerate,amsfonts,amssymb,color,graphicx,amsthm}

\usepackage[colorlinks=true, allcolors=black]{hyperref}
\usepackage{todonotes}
\usepackage[normalem]{ulem}
\numberwithin{equation}{section}
\usepackage{cite}
\usepackage{mathtools}
\usepackage[nottoc,notlot,notlof]{tocbibind}
\usepackage[toc,page]{appendix}

\usepackage{lipsum}

\newlength{\leftstackrelawd}
\newlength{\leftstackrelbwd}
\def\leftstackrel#1#2{\settowidth{\leftstackrelawd}%
	{${{}^{#1}}$}\settowidth{\leftstackrelbwd}{$#2$}%
	\addtolength{\leftstackrelawd}{-\leftstackrelbwd}%
	\leavevmode\ifthenelse{\lengthtest{\leftstackrelawd>0pt}}%
	{\kern-.5\leftstackrelawd}{}\mathrel{\mathop{#2}\limits^{#1}}}

\theoremstyle{plain}
\newtheorem{thm}{Theorem}[section]
\newtheorem{lem}[thm]{Lemma}
\newtheorem{cor}[thm]{Corollary}

\newtheorem*{thm*}{Theorem}

\theoremstyle{definition}

\newtheorem{rmk}[thm]{Remark}
\newtheorem{?}[thm]{Problem}

\renewcommand{\phi}{\varphi}
\renewcommand{\epsilon}{\varepsilon}
\makeatletter
\def\@cite#1#2{[\textbf{#1\if@tempswa , #2\fi}]}
\def\@biblabel#1{[\textbf{#1}]}
\makeatother

\makeatletter
\newcommand*{\defeq}{\mathrel{\rlap{%
			\raisebox{0.3ex}{$\m@th\cdot$}}%
		\raisebox{-0.3ex}{$\m@th\cdot$}}%
	=}
\makeatother

\makeatletter
\newcommand*{\eqdef}{=\mathrel{\rlap{%
			\raisebox{0.3ex}{$\m@th\cdot$}}%
		\raisebox{-0.3ex}{$\m@th\cdot$}}%
	}
\makeatother

\makeatother

\begin{document}
\title{Local pointwise second derivative estimates for strong solutions to the $\sigma_k$-Yamabe equation on Euclidean domains}

\author{Jonah A. J. Duncan\footnote{Mathematical Institute, University of Oxford, Andrew Wiles Building, Radcliffe Observatory Quarter, Woodstock Road, OX2 6GG, UK. Email: jonah.duncan@maths.ox.ac.uk.}~\footnote{Supported by EPSRC grant number EP/L015811/1.} ~and Luc Nguyen\footnote{Mathematical Institute and St Edmund Hall, University of Oxford, Andrew Wiles Building, Radcliffe Observatory Quarter, Woodstock Road, OX2 6GG, UK. Email: luc.nguyen@maths.ox.ac.uk.}}

\date{}

\maketitle
\begin{abstract}
	We prove local pointwise second derivative estimates for positive $W^{2,p}$ solutions to the $\sigma_k$-Yamabe equation on Euclidean domains, addressing both the positive and negative cases. Generalisations for augmented Hessian equations are also considered. 
	
	
\end{abstract}

\tableofcontents

\section{Introduction}\label{3m}

Let $\Omega\subset\mathbb{R}^n$ ($n\geq 3$) be a domain. In this paper, we obtain local pointwise second derivative estimates for positive $W^{2,p}$ solutions to the equations
\begin{equation}\label{a}\tag{1.1$^+$}
\sigma_k^{1/k}(A_u(x)) = f(x,u(x),\nabla u(x))>0, \quad  \lambda( A_u(x))\in\Gamma_k^+\quad\text{for a.e. }x\in\Omega
\end{equation}
and
\begin{equation}\label{aaa}\tag{1.1$^-$}
\sigma_k^{1/k}(-A_u(x)) = f(x,u(x),\nabla u(x))>0, \quad  \lambda(- A_u(x))\in\Gamma_k^+\quad\text{for a.e. }x\in\Omega.
\end{equation}
Throughout the paper, $A_u$ denotes the symmetric matrix-valued function 
\begin{equation*}
A_u\defeq \nabla^2 u - \frac{|\nabla u|^2}{2u}I,
\end{equation*} 
where $I$ is the $n\times n$ identity matrix and $\sigma_k$ is the $k$'th elementary symmetric polynomial, defined on a symmetric matrix $A$ with eigenvalues $\lambda(A)=(\lambda_1,\dots,\lambda_n)$ by 
\begin{equation*}
\sigma_k(A) = \sigma_k(\lambda_1,\dots,\lambda_n) \defeq \sum_{1\leq {i_1}<\dots<{i_k}\leq n} \lambda_{i_1}\cdots\lambda_{i_k}.
\end{equation*}
Note that $\sigma_1(A)$ is the trace of $A$ and $\sigma_n(A)$ is the determinant of $A$. We also denote by $\Gamma_k^+$ the open convex cone 
\begin{equation*}
\Gamma_k^+ = \{(\lambda_1,\dots,\lambda_n)\in\mathbb{R}^n:\sigma_j(\lambda_1,\dots,\lambda_n)>0~\text{for all }1\leq j\leq k\}.
\end{equation*} 

It is well-known that the equations \hyperref[a]{$(1.1^\pm)$} are elliptic. Furthermore, $\sigma_k^{1/k}$ is a concave function on the set of symmetric matrices with eigenvalues in $\Gamma_k^+$.

The motivation behind \hyperref[a]{$(1.1^\pm)$} comes from conformal geometry: if $g_{ij}=u^{-2}\delta_{ij}$ is a metric conformal to the flat metric on a domain $\Omega\subset \mathbb{R}^n$, then $uA_u$ is the $(1,1)$-Schouten tensor of $g$, and the $\sigma_k$-Yamabe equation in the so-called positive/negative ($\pm$) case is given by \setcounter{equation}{1}
\begin{equation}\label{-2}
\sigma_k(\pm uA_u) = 1, \quad \lambda(\pm A_u)\in\Gamma_k^+,  \quad u>0.
\end{equation}
The equations \eqref{-2} and their counterparts on Riemannian manifolds were first studied by Viaclovsky in \cite{Via00a}. Since then, these equations have been addressed by various authors -- for a partial list of references, see \cite{Via02, CGY02a, GW03b, LL03, Han04, LL05, Che05, GW06, Wan06, Gua07, STW07, GV07, JLL07, TW09, Li09, HLT10, GLW10, CD10, LN14, San17, GS18, CW18, BS19, AE19, BCE19, Case19, FW19, HLL19, CW19, FW20, LN20, LNW20} in the positive case and \cite{GV03b, CHY05, LS05, Guan08, GSW11, Sui17, GLN18, LN20b} in the negative case. When $k=1$, these equations reduce to the original Yamabe equation. When $k\geq 2$, they are fully nonlinear and elliptic at a solution (although, a priori, not necessarily uniformly elliptic). Fully nonlinear elliptic equations involving eigenvalues of the Hessian were first considered by Caffarelli, Nirenberg and Spruck in \cite{CNS3}. 

A priori local first and second derivative estimates play an important role in the study of the $\sigma_k$-Yamabe equation, and were established in the positive case by Chen \cite{Che05}, Guan and Wang \cite{GW03b}, Jin, Li and Li \cite{JLL07}, Li and Li \cite{LL03}, Li \cite{Li09} and Wang \cite{Wan06}. In the negative case, an a priori (global) $C^1$ estimate is proven by Gursky and Viaclovsky \cite{GV03b}, but it is unknown whether a priori $C^2$ estimates hold. In this paper, we are concerned with the local regularity of positive $W^{2,p}$ solutions to the equations \hyperref[a]{$(1.1^\pm)$}. More precisely, for $2\leq k\leq n$ we derive local pointwise boundedness of second derivatives, provided $p>kn/2$ in the positive case and $p>(k+1)n/2$ in the negative case. To simplify the discussion, we do not include the case $k=1$, in which the equations \hyperref[a]{$(1.1^\pm)$} are semilinear. We prove:

\begin{thm}\label{3}
	Let $\Omega$ be a domain in $\mathbb{R}^n$ ($n\geq 3$) and let $f\in C_{\operatorname{loc}}^{1,1}(\Omega\times(0,\infty)\times\mathbb{R}^n)$ be a positive function. Suppose that $2\leq k\leq n$, $p>kn/2$ and $u\in W_{\operatorname{loc}}^{2,p}(\Omega)$ is a positive solution to \eqref{a}. Then $u\in C_{\mathrm{loc}}^{1,1}(\Omega)$, and for any concentric balls $B_{R}\subset B_{2R}\Subset\Omega$ we have \setcounter{equation}{2}
		\begin{equation*}
		\|\nabla^2 u\|_{L^\infty(B_R)}\leq C,
		\end{equation*}
		 where $C$ is a constant depending only on $n,p,R, f$ and an upper bound for $\|\ln u\|_{W^{2,p}(B_{2R})}$. 
		\end{thm}
	
\begin{thm}\label{3'}
	Let $\Omega$ be a domain in $\mathbb{R}^n$ ($n\geq 3$) and let $f\in C_{\operatorname{loc}}^{1,1}(\Omega\times(0,\infty)\times\mathbb{R}^n)$ be a positive function. Suppose that $2\leq k\leq n$, $p>(k+1)n/2$ and $u\in W_{\operatorname{loc}}^{2,p}(\Omega)$ is a positive solution to \eqref{aaa}. Then $u\in C_{\mathrm{loc}}^{1,1}(\Omega)$, and for any concentric balls $B_{R}\subset B_{2R}\Subset\Omega$ we have
	\begin{equation*}
	\|\nabla^2 u\|_{L^\infty(B_R)}\leq C,
	\end{equation*}
	where $C$ is a constant depending only on $n,p,R, f$ and an upper bound for $\|\ln u\|_{W^{2,p}(B_{2R})}$. 
\end{thm}

\begin{rmk}
	As noted above, it is unknown whether a priori $C^2$ estimates hold for solutions to the $\sigma_k$-Yamabe equation in the negative case. We also note that for the closely related $\sigma_k$-Loewner-Nirenberg problem, there exist locally Lipschitz but non-differentiable viscosity solutions -- see \cite{LN20b}. As far as the authors are aware, Theorem \ref{3'} currently provides the only available local second derivative estimate for solutions to the $\sigma_k$-Yamabe equation in the negative case. 
\end{rmk}

To put things in perspective, we note that our estimates in Theorem \ref{3} are closely related to certain analytical aspects in the work of Chang, Gursky and Yang in \cite{CGY02a}. In \cite{CGY02a}, under natural conformally invariant conditions on a Riemannian 4-manifold $(M^4,g_0)$, the authors established the existence of a metric in the conformal class $[g_0]$ whose Schouten tensor has eigenvalues in $\Gamma_2^+$. An important part of the proof in \cite{CGY02a} was to obtain $W^{2,s}$ estimates for $4<s<5$ on smooth solutions to a one-parameter family of regularised $\sigma_2$-equations (see equation \eqref{AP2} in Appendix \ref{APPA}) which are uniform with respect to the parameter. This was achieved by first obtaining a uniform $W^{1,4}$ estimate (see Theorem 3.1 in \cite{CGY02a}), and subsequently carrying out an integrability improvement argument (see Sections 5 and 6 in \cite{CGY02a}). With the $W^{2,s}$ estimate in hand, the authors then applied a heat flow argument to obtain the desired conformal metric. 

\begin{rmk} A natural question to ask is whether the heat flow argument in \cite{CGY02a} can be avoided by instead taking the regularisation parameter directly to zero. One application of Theorem \ref{3} above and \cite[Proposition 5.3]{LN20} is that this can be achieved when $(M^4,g_0)$ is locally conformally flat. We refer the reader to Appendix \ref{APPA} for the details. 
\end{rmk}

 Our work is also closely related to the work of Urbas in \cite{Urb00}, where local pointwise second derivative estimates for $W^{2,p}$ solutions to the $k$-Hessian equation
\begin{equation*}
\sigma_k^{1/k}(\nabla^2 u(x)) = f(x)>0, \quad \lambda(\nabla^2 u(x))\in\Gamma_k^+
\end{equation*}
were established on domains in $\mathbb{R}^n$. At the heart of Urbas' proof is also an integrability improvement argument, assuming an initial lower bound of $p>kn/2$ (see also \cite{Urb01, Urb88, CM17, BL, Urb07}). By an application of Moser iteration, the $C^{1,1}_{\operatorname{loc}}$ estimate is then obtained. We note that Moser iteration has previously been utilised in the context of the $\sigma_k$-Yamabe equation to establish local boundedness of solutions, see for instance \cite{Han04, Gonz05, Gonz06}.

We will in fact prove a more general version of Theorems \ref{3} and \ref{3'}, and consider an operator of the form
\begin{equation}\label{aa}
A_H[u] \defeq \nabla^2 u - H[u]
\end{equation}
in place of $\pm A_u$. Here, $H[u](x)=H(x,u(x),\nabla u(x))$ for a given matrix-valued function $H=H(x,z,\xi)\in C_{\operatorname{loc}}^{1,1}(\Omega\times\mathbb{R}\times\mathbb{R}^n\,;\, \operatorname{Sym}_n(\mathbb{R}))$, where $\operatorname{Sym}_n(\mathbb{R})$ denotes the space of real symmetric $n\times n$ matrices. Rather than \hyperref[a]{$(1.1^\pm)$}, we consider the equation 
\begin{equation}\label{7}
\sigma_k^{1/k}\big(A_H[u](x)\big) = f(x,u(x),\nabla u(x))>0, \quad \lambda(A_H[u](x)) \in\Gamma_k^+  \quad \mathrm{for~a.e.~}x\in\Omega,
\end{equation}
where $f\in C^{1,1}_{\operatorname{loc}}(\Omega\times\mathbb{R}\times\mathbb{R}^n)$. 

It is clear that if $u$ satisfies \eqref{a} with $u\geq \frac{1}{C}>0$, then $u$ satisfies \eqref{7} provided $H(x,z,\xi) = \frac{|\xi|^2}{2z}I$ for $z\geq\frac{1}{C}$. Likewise, if $u\in W^{2,p}_{\operatorname{loc}}(\Omega)$ is a solution to \eqref{aaa} with right hand side (RHS) $f$ and $u\geq \frac{1}{C}>0$, then $w\defeq - u \in W_{\operatorname{loc}}^{2,p}(\Omega)$ satisfies \eqref{7} with RHS $\tilde{f}(x,z,\xi) \defeq f(x,-z,-\xi)$, provided $H(x,z,\xi) = \frac{|\xi|^2}{2z}I$ for $z\leq-\frac{1}{C}$. Therefore, for the purpose of obtaining Theorems \ref{3} and \ref{3'}, it will suffice to consider the case that $H$ is a multiple of the identity matrix:

\begin{thm}\label{8}
	Let $\Omega$ be a domain in $\mathbb{R}^n$ ($n\geq 3$), $f\in C_{\operatorname{loc}}^{1,1}(\Omega\times\mathbb{R}\times\mathbb{R}^n)$ a positive function and $H\in C_{\operatorname{loc}}^{1,1}(\Omega\times\mathbb{R}\times\mathbb{R}^n\,;\,\operatorname{Sym}_n(\mathbb{R}))$. Suppose $2\leq k\leq n$, $p\geq 1$ and $u\in W_{\operatorname{loc}}^{2,p}(\Omega)$ is a solution to \eqref{7}, and that one of the following conditions holds:
	\begin{enumerate}
		\item $H(x,z,\xi) = H_1(x,z)|\xi|^2 I$ with $H_1\geq 0$ and $p>\frac{kn}{2}$,
		\item $H(x,z,\xi) = H_2(x,z,\xi)I$ and $p>\frac{(k+1)n}{2}$. 
	\end{enumerate}
	Then $u\in C^{1,1}_{\operatorname{loc}}(\Omega)$, and for any concentric balls $B_{R}\subset B_{2R} \Subset \Omega$ we have
	\begin{equation}\label{63}
	\|\nabla^2 u\|_{L^\infty(B_R)}\leq C,
	\end{equation}
	where $C$ is a constant depending only on $n,p,R, f, H$ and an upper bound for $\|u\|_{W^{2,p}(B_{2R})}$. 
\end{thm} 

\begin{rmk}\label{64}
	The constant $C$ in \eqref{63} depends only on $n,p,R$ and upper bounds for $\| u\|_{W^{2,p}(B_{2R})}$, $\|H\|_{C^{1,1}(\Sigma)}$ and $ \|\ln f\|_{C^{1,1}(\Sigma)}$, where $\Sigma \defeq \overline{B}_{2R}\times[-M,M]\times \overline{B}_{M}(0)\subset\Omega\times\mathbb{R}\times\mathbb{R}^n$ and $M \geq \|u\|_{C^1(\overline{B}_{2R})}$. Note that since $p>n$ in Theorem \ref{8}, an upper bound for $\| u\|_{W^{2,p}(B_{2R})}$ implies an upper bound for $\| u\|_{C^1(\overline{B}_{2R})}$, in light of the Morrey embedding theorem. 
\end{rmk}
\begin{rmk}
	When $H\equiv 0$ and $f=f(x)$, Theorem \ref{8} was proved in \cite[Theorem 1.6]{Urb00}.
\end{rmk}

The matrix $A_H[u]$ introduced in \eqref{aa} is sometimes referred to as an augmented Hessian of $u$. The corresponding augmented Hessian equations have been extensively studied in recent years -- see \cite{JT17, JT18, JT19} and the references therein. In this vein, it is therefore of interest to generalise Theorem \ref{8} to arbitrary $H\in C_{\operatorname{loc}}^{1,1}$. As we will see, the proof of Theorem \ref{8} uses some favourable divergence structure in the case that $H$ is a multiple of the identity matrix. However, when $k=2$, a similar divergence structure holds for general $H$ and we obtain the following:

\begin{thm}\label{56}
		Let $\Omega$ be a domain in $\mathbb{R}^n$ ($n\geq 3$), $f\in C_{\operatorname{loc}}^{1,1}(\Omega\times\mathbb{R}\times\mathbb{R}^n)$ a positive function and $H\in C_{\operatorname{loc}}^{1,1}(\Omega\times\mathbb{R}\times\mathbb{R}^n\,;\,\operatorname{Sym}_n(\mathbb{R}))$. Suppose $p>\frac{3n}{2}$ and $u\in W_{\operatorname{loc}}^{2,p}(\Omega)$ is a solution to \eqref{7} with $k=2$. Then $u\in C^{1,1}_{\mathrm{loc}}(\Omega)$, and for any concentric balls $B_{R}\subset B_{2R} \Subset \Omega$ we have
		\begin{equation*}
		\|\nabla^2 u\|_{L^\infty(B_R)}\leq C,
		\end{equation*}
		where $C$ is a constant depending only on $n,p,R,f,H$ and an upper bound for $\| u\|_{W^{2,p}(B_{2R})}$.
	\end{thm} 

\begin{rmk}
	In \cite{JT17, JT18, JT19} and the references therein, it is usually assumed that $H$ satisfies a so-called co-dimension one convexity condition, which is known to be a necessary and sufficient condition to obtain $C^1$ estimates -- see \cite{MTW, TW09b, Loe}. We point out that we do \textit{not} assume a co-dimension one convexity condition in our treatment of second derivative estimates (the exception is Case 1 of Theorem \ref{8}, where we have convexity in $\xi$). 
\end{rmk}

Under a stronger assumption on $p$, we will also obtain an extension of Theorem \ref{56} to the case $k\geq 3$ -- see Section \ref{BA}.

 In adapting the methods of \cite{Urb00} to prove Theorems \ref{8} and \ref{56}, we will need to deal with the term $H[u]$ which, whilst being of lower order in the definition of $A_H[u]$, creates new higher order terms in our estimates. Roughly speaking, the two terms which are formally problematic consist of:
 \begin{itemize}
 	\item[\textbf{(i)}] a contraction of the linearised operator 
 	\begin{equation}\label{b}
 	F[u]^{ij} \defeq \frac{\partial\sigma_k(A_H[u])}{\partial (A_H[u])_{ij}}
 	\end{equation}
 	with double difference quotients of $H[u]_{ij}$ (this arises as a result of taking difference quotients of \eqref{7} twice), and
 	\item[\textbf{(ii)}] the divergence of $F[u]^{ij}$ multiplied by a term formally of third order in $u$ (this arises after integrating by parts). 
 \end{itemize}
In \cite{Urb00}, neither of these terms exist since $F[u]^{ij}$ is divergence-free when $H\equiv 0$. In the more general case that we are considering, it is unclear whether these third order terms have a favourable sign individually. However, we will estimate them so as to show that, when combined, they yield a cancellation phenomenon that ensures the overall higher order contribution is positive. For the estimates of the higher order terms arising from the divergence of $F[u]^{ij}$, see Lemmas \ref{83} and \ref{83'}, and for those arising from the double difference quotients of $H[u]$, see Lemma \ref{71}. For the resulting cancellation phenomena, see Corollaries \ref{J}, \ref{M} and \ref{N}.

 We close the introduction by noting that in Theorems \ref{3} and \ref{3'}, we do not know whether our lower bounds on $p$ to obtain $C^{1,1}_{\operatorname{loc}}$ regularity are sharp, and it would be interesting to determine the sharp lower bounds. In the case of the $k$-Hessian equation for $3\leq k\leq n$, it is shown by Urbas in \cite{Urb90} that there exist $W^{2,p}$-strong solutions with $p<\frac{k(k-1)}{2}$ which fail to be $C^{1,\alpha}_{\operatorname{loc}}$ for any $\alpha>1-\frac{2}{k}$. Other lower bounds on $p$ leading to $C^{1,1}_{\operatorname{loc}}$ regularity for $k$-Hessian equations have been studied in \cite{Urb01, Urb88, CM17, BL, Urb07}, for instance. \newline

The plan of the paper is as follows. We begin in Section \ref{79} with an outline of the proof of Theorems \ref{8} and \ref{56}. This prompts us to consider the divergence structure of the linearised operator, which we address in Section \ref{80}, and also motivate the estimates established from Section \ref{81} onwards. In Section \ref{81} we carry out the main body of our integral estimates. In Section \ref{82}, we use these estimates and the Moser iteration technique to obtain the desired $C^{1,1}_{\operatorname{loc}}$ estimates, completing the proofs of Theorems \ref{8} and \ref{56}. In Section \ref{BA}, we give the aforementioned extension of Theorem \ref{56} to the case $k\geq 3$.

\section{Outline of the proofs of Theorems \ref{8} and \ref{56}}\label{79}

Our proofs of Theorems \ref{8} and \ref{56} use an integrability improvement argument, from which the $C^{1,1}_{\operatorname{loc}}$ estimate is obtained by the Moser iteration technique. In Case 1 of Theorem \ref{8}, we will obtain, for a solution $u\in W_{\operatorname{loc}}^{2,q+k-1}(\Omega)$ to \eqref{7} with $q>\frac{kn}{2}-k+1$, the estimate 
\begin{equation}\label{AB}
\bigg(\int_{B_{R+\rho}}(\Delta u+C_1)^{\beta q}\bigg)^{1/\beta}\leq \frac{Cq}{\rho^2}\int_{B_{R+3\rho}}(\Delta u+C_1)^{q+k-1},
\end{equation}
where $\rho\in (0,\frac{R}{3}]$, $\beta = \frac{kn}{kn-2k+2}$ and $C_1$ is a positive constant ensuring $\Delta u + C_1\geq 1$ a.e. (see the paragraph after Remark \ref{101} for the justification of the existence of $C_1$).  Similarly, in Case 2 of Theorem \ref{8} and in Theorem \ref{56}, we will obtain, for a solution $u\in W_{\operatorname{loc}}^{2,q+k}(\Omega)$ to \eqref{7} with $q>\frac{(k+1)n}{2}-k$, the estimate 
\begin{equation}\label{AB'}
\bigg(\int_{B_{R+\rho}}(\Delta u+C_1)^{\beta q}\bigg)^{1/\beta}\leq \frac{Cq}{\rho^2}\int_{B_{R+3\rho}}(\Delta u+C_1)^{q+k},
\end{equation}
now with $\beta = \frac{(k+1)n}{(k+1)n  - 2(k+1)+2}$. The estimates \eqref{AB} and \eqref{AB'} then yield an improvement in integrability under the respective lower bounds on $q$, which can then be iterated to yield the desired $C^{1,1}_{\operatorname{loc}}$ estimates.\footnote{One might ask whether a reverse H\"older-type inequality for a single second derivative $\nabla_l\nabla_l u$, similar to \eqref{AB} and \eqref{AB'}, can be established. We have been unable to show this.}

In the rest of this section we explain how the estimates \eqref{AB} and \eqref{AB'} are obtained. Due to the lack of regularity, we derive our estimates through taking difference quotients of the equation \eqref{7}. For an index $l\in\{1,\dots,n\}$ and increment $h\in\mathbb{R}\backslash\{0\}$, we recall the first order difference quotient $\nabla^h_l u(x) \defeq h^{-1}(u(x+he_l)-u(x))$ and the second order difference quotient
\begin{equation}\label{s4}
\ \Delta_{ll}^h u(x) \defeq \nabla_l^h(\nabla^{-h}_l u(x)) = \frac{u(x+he_l)-2u(x) + u(x-he_l)}{h^2}.
\end{equation}
We also denote
\begin{equation*}
v_h(x) \defeq \sum_{l=1}^n \Delta_{ll}^h u(x).
\end{equation*}
The above expressions are well-defined for $x\in\Omega_h\defeq\{y\in\Omega:\operatorname{dist}(y,\partial\Omega)>|h|\}$. 

It is well-known (see, for instance, \cite[Lemma 7.23]{GT}) that 
\begin{equation}\label{60}
\|\nabla_l^h u\|_{L^s(\Omega')} \leq \|\nabla_l u \|_{L^s(\Omega)} \quad\text{for all } s\geq 1 \text{ and } \Omega'\Subset\Omega \text{ s.t. }\operatorname{dist}(\Omega',\partial\Omega)>|h|.
\end{equation}
It follows from \eqref{s4} and \eqref{60} that there exists a constant $C=C(n)$ such that 
\begin{equation}\label{3g'}
\|v_h\|_{L^s(\Omega')} \leq C\|\nabla^2 u \|_{L^s(\Omega)}\quad \text{for all }s\geq 1.
\end{equation}
We will also use the following fact -- see Appendix \ref{APPB} for a proof:

\begin{lem}\label{3g}
Suppose $u \in W^{2,s}(\Omega)$ for some $s\geq1$. Then $v_h \rightarrow \Delta u$ in $L^s_{\mathrm{loc}}(\Omega)$ as $h\rightarrow 0$.
\end{lem}

We assume now that both the increment $h$ and our solution $u$ are fixed, and write $v$ as shorthand for $v_h$. Taking difference quotients of the equation $\sigma_k^{1/k}(A_H[u](x))=f[u](x)\defeq f(x,u(x),\nabla u(x))$ and appealing to the concavity of $\sigma_k^{1/k}$ in $\Gamma_k^+$, we will derive (at the start of Section \ref{81}) the pointwise estimate
\begin{align}\label{3a'}
\sum_l k(f[u])^{k-1}\Delta_{ll}^h f[u]& \leq F[u]^{ij}\nabla_i\nabla_j v - \sum_l F[u]^{ij}\Delta_{ll}^h (H[u])_{ij}\quad\mathrm{a.e.~in~}\Omega_h. 
\end{align}
Here, $F[u]^{ij}=\partial\sigma_k(A_H[u])/\partial (A_H[u])_{ij}$ is the linearised operator.

\begin{rmk}\label{100}
	 In \eqref{3a'}, and from this point onwards, summation notation is employed \textit{only over repeated indices which appear in both upper and lower positions}. Positioning of indices is purely to indicate whether summation convention is being utilised; since we are working with the Euclidean metric, we are free to raise and lower indices at will. For instance, $A_{ij}$, $A^i_j$, $A^j_i$ and $A^{ij}$ all denote the $(i,j)$-entry of a symmetric matrix $A$. Similarly, we do not distinguish between the derivatives $\nabla^i$ and $\nabla_i$ when using index notation. 
\end{rmk}
\begin{rmk}\label{101}
	Since $u$ is fixed, we write $f[u], H[u], A_H[u], F[u]^{ij}$ etc. to emphasise that these are to be considered as functions of $x$. If it is clear from the context (e.g. if there are no derivatives involved), we will simply write $f, H, A_H, F^{ij}$ etc. 
\end{rmk}

The estimates \eqref{AB} and \eqref{AB'} are derived by testing \eqref{3a'} against suitable test functions. First fix a ball $B_{2R}\Subset \Omega_h$. Since $\lambda(A_H)\in\Gamma_2^+$ is equivalent to $\operatorname{tr}(A_H) = \Delta u - \operatorname{tr}(H)>0$ and $\sigma_2(A_H)>0$, there exists a constant $C_1\geq 0$ (depending on an upper bound for $ \|H\|_{C^0(\Sigma)}$ - see Remark \ref{64}) for which $\Delta u + C_1\geq 1$ and $|\nabla^2 u| \leq \Delta u + C_1$ a.e. in $B_{2R}$. We define $\tilde{v} \defeq v+C_1$, and for a small parameter $\delta>0$ (that we eventually take to zero) we denote
\begin{equation*}
Q_\delta \defeq \big((\tilde{v}^+)^2 + \delta^2 \big)^{1/2}. 
\end{equation*}
For $\rho\in(0,\frac{R}{3}]$ we also let $\eta\in C_{\mathrm{c}}^\infty(B_{R+2\rho})$ be a standard non-negative cutoff function. Testing \eqref{3a'} against $\eta Q_\delta^{q-1}$ (where $q>1$) then yields
\begin{align}\label{11''}
\sum_l\int_{B_{R+2\rho}}k\eta Q_\delta^{q-1}f^{k-1}\Delta_{ll}^h f[u] & \leq \int_{B_{R+2\rho}}\eta Q_\delta^{q-1}F^{ij}\nabla_i\nabla_j \tilde{v} - \sum_l\int_{B_{R+2\rho}}\eta Q_\delta^{q-1}F^{ij}\Delta_{ll}^h (H[u])_{ij}
\end{align}
for all $q>1$ and $u\in W_{\operatorname{loc}}^{2,q+k-1}(\Omega)\cap W_{\operatorname{loc}}^{1,\infty}(\Omega)$ solving \eqref{7}.

For ease of outlining our argument, let us suppose that $f=f(x,z)$ (the general case $f=f(x,z,\xi)$ will only require minor changes - see Section \ref{AT}). Then the integrand on the left hand side (LHS) of \eqref{11''} is a lower order term, whereas the integrands on the RHS of \eqref{11''} involve higher order terms, formally of fourth and third order in the limit $h\rightarrow 0$, and thus need to be treated. 

In Section \ref{81}, we integrate by parts in the first integral on the RHS of \eqref{11''}, using a result of Section \ref{80} that tells us $\nabla_iF[u]^{ij}$ is a regular distribution belonging to $L_{\operatorname{loc}}^{(q+k-1)/(k-1)}(\Omega)$ if $u\in W_{\operatorname{loc}}^{2,q+k-1}(\Omega)\cap W_{\operatorname{loc}}^{1,\infty}(\Omega)$. After taking $\delta\rightarrow 0$ and carrying out some further calculations (see Lemmas \ref{P} and \ref{21'}), we will obtain the estimate 
\begin{align}\label{57}
&\frac{q-1}{Cq^2}\int_{B_{R+\rho}}f^k\frac{\big|\nabla\big((\tilde{v}^+)^{q/2}\big)\big|^2}{\operatorname{tr}(A_H)}  + \int_{B_{R+2\rho}}\eta(\tilde{v}^+)^{q-1}\nabla_i F[u]^{ij}\nabla_j\tilde{v} \nonumber \\
& \qquad  + \sum_l \int_{B_{R+2\rho}}\eta(\tilde{v}^+)^{q-1} F^{ij}\Delta_{ll}^h (H[u])_{ij} \nonumber \\
& \qquad  \qquad \leq \frac{C}{\rho^2}\bigg(\int_{B_{R+2\rho}}(\tilde{v}^+)^{q+k-1} + \int_{B_{R+3\rho}}(\Delta u + C_1)^{q+k-1}\bigg),
\end{align}
where $C$ is a constant independent of $h, q$ and $\rho$. 

Whilst the first integral on the LHS of \eqref{57} is a favourable positive higher order term, the other two integrals on the LHS (which we denote by $\operatorname{(I_2)}_h$ and $\operatorname{(I_3)}_h$, respectively) involve higher order terms which are, a priori, of unknown sign. Treating $\operatorname{(I_2)}_h$ and $\operatorname{(I_3)}_h$ is the most technical part of our proof.

Now, if we momentarily assume sufficiently high regularity on $u$, say $u\in W_{\operatorname{loc}}^{2,q+2k-1}(\Omega)\cap W_{\operatorname{loc}}^{1,\infty}(\Omega)$ ($q>1$), the issue of dealing with $\operatorname{(I_2)}_h$ and $\operatorname{(I_3)}_h$ is largely simplified. As will be detailed in the proof of Theorem \ref{BB}, one may apply the Cauchy inequality to each of the integrands and absorb the resulting third order terms into the positive term on the LHS of \eqref{57}. Under the stated integrability assumption, this crude estimation is sufficient to show
\begin{equation*}
\frac{q-1}{q^2}\int_{B_{R+\rho}}f^k\frac{\big|\nabla\big((\tilde{v}^+)^{q/2}\big)\big|^2}{\operatorname{tr}(A_H)} \leq \frac{C}{\rho^2}\bigg(\int_{B_{R+2\rho}}(\tilde{v}^+)^{q+2k-1} + \int_{B_{R+3\rho}}(\Delta u + C_1)^{q+2k-1}\bigg).
\end{equation*}
An estimate analogous to \eqref{AB} and \eqref{AB'} can then be obtained, assuming $q>kn-2k+1$. 

The difficulty is to therefore deal with $\operatorname{(I_2)}_h$ and $\operatorname{(I_3)}_h$ under the \textit{weaker} integrability assumptions of Theorems \ref{8} and \ref{56}. At this point, we make the distinction between the various cases. In each case,  we estimate $\operatorname{(I_2)}_h$ and $\operatorname{(I_3)}_h$ so as to produce a cancellation phenomenon when combined, leaving only lower order terms; see Lemmas \ref{83} and \ref{83'} for the estimates on $\operatorname{(I_2)}_h$, Lemma \ref{71} for the estimates on $\operatorname{(I_3)}_h$, and Corollaries \ref{J}, \ref{M} and \ref{N} for the resulting cancellations. It will then follow from \eqref{57} that, in Case 1 of Theorem \ref{8} with the relaxed assumption $u\in W_{\operatorname{loc}}^{2,q+k-1}(\Omega)\cap W_{\operatorname{loc}}^{1,\infty}(\Omega)$ ($q>1$), we have the estimate
\begin{equation}\label{67}
\frac{q-1}{q^2}\int_{B_{R+\rho}}f^k\frac{\big|\nabla\big((\tilde{v}^+)^{q/2}\big)\big|^2}{\operatorname{tr}(A_H)} \leq \frac{C}{\rho^2}\bigg(\int_{B_{R+2\rho}}(\tilde{v}^+)^{q+k-1} + \int_{B_{R+3\rho}}(\Delta u+C_1)^{q+k-1}\bigg).
\end{equation}
Similarly, in the remaining cases with $u\in W_{\operatorname{loc}}^{2,q+k}(\Omega)\cap W_{\operatorname{loc}}^{1,\infty}(\Omega)$ ($q>1$), we will obtain 
\begin{equation}\label{67'}
\frac{q-1}{q^2}\int_{B_{R+\rho}}f^k\frac{\big|\nabla\big((\tilde{v}^+)^{q/2}\big)\big|^2}{\operatorname{tr}(A_H)} \leq \frac{C}{\rho^2}\bigg(\int_{B_{R+2\rho}}(\tilde{v}^+)^{q+k}+\int_{B_{R+3\rho}}(\Delta u+C_1)^{q+k}\bigg).
\end{equation}

To obtain \eqref{AB} from \eqref{67} (resp. \eqref{AB'} from \eqref{67'}), we proceed as follows  (the details can be found in Section \ref{82}). We first obtain an integral estimate for $\big|\nabla\big((\tilde{v}^+)^{q/2}\big)\big|^2$, to which we can apply the Sobolev inequality. We then justify taking the limit $h\rightarrow 0$ and impose the lower bound $q+k-1> \frac{kn}{2}$ (resp. $q+k>\frac{(k+1)n}{2}$), from which we obtain \eqref{AB} (resp. \eqref{AB'}).

\section{Divergence structure of the linearised operator $F[u]^{ij}$}\label{80}

In this section we derive a divergence formula for the linearised operator $F[u]^{ij}$ (defined in \eqref{b}), which we will use at various stages of our proof. 

We note that in the case that $A_H[u]=\nabla^2 u$ or $A_H[u] = A_u$, the divergence properties of $F[u]^{ij}$ are well-documented (for smooth $u$). In the former case, $F[u]^{ij}$ is divergence-free with respect to the flat metric (see \cite{Rei73}), and in the latter case, $u^{1-k}F[u]^{ij}$ is divergence-free with respect to the conformal metric $g_{ij}=u^{-2}\delta_{ij}$ (see \cite{Via00a}). For related discussions, see also \cite{GW03a, BV04, Han06, STW07, BG08}. 

For $A\in \operatorname{Sym}_n(\mathbb{R})$ and $1\leq k\leq n$, define the $k$'th Newton tensor of $A$ inductively by
\begin{equation}\label{28}
T_k(A) \defeq \sigma_k(A)I - T_{k-1}(A) A, \quad T_0(A)^{ij} \defeq \delta^{ij}. 
\end{equation}
It is well-known (see \cite{Rei73}) that
\begin{equation}\label{29}
\frac{\partial\sigma_k(A)}{\partial A_{ij}} = T_{k-1}(A)^{ij}
\end{equation}
and
\begin{equation}\label{1''}
\operatorname{tr}(T_k(A)) = (n-k)\sigma_k(A),
\end{equation}
and moreover $T_{k-1}(A)^{ij}$ is positive definite when $\lambda(A)\in\Gamma_k^+$ (see \cite{CNS3}). In particular, by \eqref{b} and \eqref{29}, $F[u]^{ij} = T_{k-1}(A_H[u])^{ij}$. 

\begin{lem}\label{12}
	Let $\Omega\subset\mathbb{R}^n$ be a domain and $u\in C^3(\Omega)$. Then for $H\in C^1(\Omega\times \mathbb{R}\times\mathbb{R}^n\,;\,\operatorname{Sym}_n(\mathbb{R}))$ and $2\leq k \leq n$,
	\begin{equation}\label{15}
	\nabla_i F[u]^{ij} =  \sum_{p=1}^{k-1} (-1)^{p+1}T_{k-p-1}(A_H)^{ab}\Big(\nabla_a (H[u])^c_b - \nabla^c (H[u])_{ab}\Big)(A_H^{p-1})^j_c \eqdef V[u]^j.
	\end{equation}
	Moreover, if $H(x,z,\xi) = H_2(x,z,\xi)I$, then
	\begin{equation}\label{15'}
	\nabla_i F[u]^{ij} = -(n-k+1)\nabla_i (H_2[u])\, T_{k-2}(A_H)^{ij}.
	\end{equation}
\end{lem}

\begin{proof}
	The identity \eqref{15} will follow once we show that for $1\leq k\leq n-1$,
	\begin{align}\label{-1}
	\nabla_i T_k(A_H[u])^{ij} = \sum_{p=1}^k (-1)^{p+1} T_{k-p}(A_H)^{ab}\Big(\nabla_a (H[u])^c_b - \nabla^c (H[u])_{ab}\Big)(A_H^{p-1})^j_c \quad\mathrm{in~}\Omega.
	\end{align}
	Similarly, \eqref{15'} will follow once we show that for $1\leq k\leq n-1$ and $H(x,z,\xi)=H_2(x,z,\xi)I$, 
	\begin{equation}\label{17}
	\nabla_i T_k(A_H[u])^{ij} = -(n-k)\nabla_i (H_2[u])\,T_{k-1}(A_H)^{ij}\quad\mathrm{in~}\Omega. 
	\end{equation}
	
	To this end, we take the divergence of both sides in \eqref{28}, which yields
	\begin{align}\label{18}
	\nabla_i T_k(A_H[u])^{ij} & = \nabla^j\sigma_k(A_H[u]) - \nabla_i \big(T_{k-1}(A_H[u])^{il}(A_H[u])^j_l\big) \nonumber \\
	& = \frac{\partial\sigma_k(A_H)}{\partial(A_H)_{il}}\nabla^j (A_H[u])_{il} - \nabla_i(T_{k-1}(A_H[u]))^{il}(A_H)^j_l - T_{k-1}(A_H)^{il}\nabla_i (A_H[u])^j_l\nonumber \\
	& \leftstackrel{\eqref{29}}{=} T_{k-1}(A_H)^{il}\big(\nabla^j(A_H[u])_{il}-\nabla_i(A_H[u])^j_l\big) - \nabla_i(T_{k-1}(A_H[u]))^{il}(A_H)^j_l \nonumber \\
	& = T_{k-1}(A_H)^{il}\big(\nabla_i (H[u])_l^j - \nabla^j (H[u])_{il}\big) - \nabla_i\big(T_{k-1}(A_H[u])\big)^{il}(A_H)^j_l. 
	\end{align}
	Then \eqref{-1} is readily seen by applying \eqref{18} iteratively. 
	
	We now turn to \eqref{17}, for which we apply an induction argument on $k$ using \eqref{18}. The base case $k=1$ is clear. We  suppose that for some $k\geq 2$ we have the identity
	\begin{equation}\label{94}
	\nabla_i T_{k-1}(A_H[u])^{ij} = -(n-k+1)\nabla_i (H_2[u])\,T_{k-2}(A_H)^{ij},
	\end{equation}
	and we show that \eqref{17} then follows. First observe that, by \eqref{94} and the fact $H_{ij} = H_2\delta_{ij}$, \eqref{18} simplifies to
	\begin{align}\label{11a}
	\nabla_i T_k(A_H[u])^{ij}  & =\nabla_i (H_2[u]) T_{k-1}(A_H)^{ij}  -\nabla^j (H_2[u]) \operatorname{tr}(T_{k-1}(A_H)) \nonumber \\
	& \quad + (n-k+1)\nabla_i (H_2[u]) (T_{k-2}(A_H) A_H)^{ij}. 
	\end{align}
	After substituting \eqref{28} and \eqref{1''} into the last term and the penultimate term in \eqref{11a}, respectively, we arrive at \eqref{17}.
\end{proof}

Note that $V[u]^j$ (defined in \eqref{15}) contains at most second order derivatives of $u$. As a consequence, $\nabla_i F[u]^{ij}$ is a regular distribution for $u\in W_{\operatorname{loc}}^{2,q+k-1}(\Omega)\cap W_{\operatorname{loc}}^{1,\infty}(\Omega)$. More precisely, we have:

\begin{lem}\label{22}
	Let $\Omega\subset\mathbb{R}^n$ be a domain and $u\in W_{\operatorname{loc}}^{2,q+k-1}(\Omega)\cap W_{\operatorname{loc}}^{1,\infty}(\Omega)$ with $q>1$ and $2\leq k\leq n$. Then for $H\in C_{\operatorname{loc}}^{0,1}(\Omega\times \mathbb{R}\times\mathbb{R}^n;\operatorname{Sym}_n(\mathbb{R}))$ and $\phi\in W_0^{1,s}(\Omega;\mathbb{R}^n)$, $s\defeq\frac{q+k-1}{q}$, we have
	\begin{equation}\label{25}
	\int_{\Omega} F[u]^{ij}\nabla_i\phi_j = -\int_{\Omega}V[u]^j\phi_j,
	\end{equation}
	where $V[u]^j$ is defined in \eqref{15}. In particular, $\nabla_i F[u]^{ij} = V[u]^j\in L_{\operatorname{loc}}^{(q+k-1)/(k-1)}(\Omega)$ and
	\begin{equation}\label{3h}
	\big|\nabla_i F[u]^{ij}\big| \leq C\big(1+ |\nabla^2 u |^{k-1} \big)\quad \mathrm{a.e.~in~}B_{2R},
		\end{equation}
		where $C$ is a constant depending on an upper bound for $\|H\|_{C^{0,1}(\Sigma)}$. 
\end{lem}
\begin{proof}
	It is clear that $u\in W_{\operatorname{loc}}^{2,q+k-1}(\Omega)\cap W_{\operatorname{loc}}^{1,\infty}(\Omega)$ implies $V[u]^j \in L_{\operatorname{loc}}^{(q+k-1)/(k-1)}(\Omega)$. Since $\frac{1}{s} + \frac{k-1}{q+k-1}=1$, it suffices to prove \eqref{25} for $\phi\in C_0^\infty(\Omega;\mathbb{R}^n)$. Let $u_{(m)}\in C^3(\Omega)$ be such that $u_{(m)}\rightarrow u$ in $W_{\operatorname{loc}}^{2,q+k-1}(\Omega)$. Then by \eqref{15}, we have for each $m\in\mathbb{N}$ the identity $\nabla_i F[u_{(m)}]^{ij} = V[u_{(m)}]^j$, and it follows that 
	\begin{equation}\label{24}
	\int_{\Omega} F[u_{(m)}]^{ij}\nabla_i\phi_j = -\int_{\Omega} V[u_{(m)}]^j\phi_j. 
	\end{equation}
	Now, since $u_{(m)}\rightarrow u$ in $W_{\operatorname{loc}}^{2,q+k-1}(\Omega)$, we have both $F^{ij}[u_{(m)}]\rightarrow F^{ij}[u]$ and $V[u_{(m)}]\rightarrow V[u]$ in $L_{\operatorname{loc}}^{(q+k-1)/(k-1)}(\Omega)$. In particular, we can take $m\rightarrow \infty$ in \eqref{24} to get \eqref{25}. The estimate \eqref{3h} follows from the definition of $V[u]^j$.
\end{proof}

\section{Main estimates}\label{81}

In this section we prove our main estimates, which will then be used in the proof of our main results in Section \ref{82}. Largely, our estimates will be concerned with terms involving the contraction of the linearised operator $F=(F^{ij})$ and its divergence with various other tensors, such as $\nabla^2 \tilde{v}$, $\nabla \tilde{v}$ and $(\Delta^h_{ll}H[u]_{ij})$.

\subsection{Initial integral estimates: isolating higher order terms}\label{O}

The following lemma provides the starting point for our integral estimates: 

\begin{lem}\label{CC}
	Suppose $f\in C^0(\Omega\times\mathbb{R}\times\mathbb{R}^n)$ is positive, $H\in C^0(\Omega\times\mathbb{R}\times\mathbb{R}^n\,;\,\operatorname{Sym}_n(\mathbb{R}))$ and $u$ is a solution to \eqref{7}. Then for fixed $h$, 
	\begin{align}\label{3a}
	\sum_l kf^{k-1}\Delta_{ll}^h f[u]& \leq  F^{ij}\nabla_i\nabla_j v - \sum_l F^{ij}\Delta_{ll}^h (H[u])_{ij} \quad \mathrm{a.e.~in~}\Omega_h.
	\end{align}
\end{lem}
\begin{proof}
	The proof follows \cite{Urb00}, with some adjustments. For $A\in \operatorname{Sym}_n(\mathbb{R})$, let  $G^{ij}(A) = \partial\sigma_k^{1/k}(A)/\partial A_{ij} = k^{-1}\sigma_k(A)^{(1-k)/k}F^{ij}(A)$, and denote $G^{ij} \defeq G^{ij}(A_H[u])$. Fix $l\in\{1,\dots,n\}$ and $h\in\mathbb{R}\backslash\{0\}$. Then there exists a set $S_{h,l}\subset\Omega_h$ with $\mathcal{L}(\Omega_h\backslash S_{h,l})=0$ (where $\mathcal{L}$ is the Lebesgue measure) such that $\lambda(A_H[u](x)), \lambda(A_H[u](x\pm he_l))\in\Gamma_k^+$ for all $x\in S_{h,l}$. By concavity of $\sigma_k^{1/k}$ in $\Gamma_k^+$, it follows that for $x\in S_{h,l}$ we have 
	\begin{equation}\label{27}
	\sigma_k^{1/k}\big(A_H[u](x\pm he_l)\big) - \sigma_k^{1/k}(A_H[u](x))  \leq  G^{ij}(x)\big(A_H[u](x\pm he_l)- A_H[u](x)\big)_{ij}.
	\end{equation}
	Adding the two equations in \eqref{27}, dividing through by $h^2$ and summing over $l$, we have
	\begin{equation}\label{3d}
	\sum_l \Delta_{ll}^h \sigma_k^{1/k}\big(A_H[u](x)\big) \leq \sum_l G^{ij}(x) \Delta^h_{ll}\big(A_H[u](x)\big)_{ij}\quad \text{for all }x\in S_{h} \defeq \bigcap_{l=1}^n S_{h,l}, 
	\end{equation}
	 with $S_h$ clearly satisfying $\mathcal{L}(\Omega_h\backslash S_{h})=0$. Substituting the definition of $G^{ij}$ into \eqref{3d} and recalling that $A_H[u] = \nabla^2 u - H[u]$, we obtain
	\begin{equation}\label{3c}
	\sum_l k\sigma_k^{\frac{k-1}{k}}(A_H)\Delta_{ll}^h \sigma_k^{1/k}(A_H[u]) \leq \sum_l  F^{ij} \Delta_{ll}^h\big(\nabla^2 u - H[u]\big)_{ij}\quad\mathrm{in~}S_h. 
	\end{equation}
	Substituting the equation $\sigma_k^{1/k}(A_H) = f$ into the LHS of \eqref{3c}, and commuting difference quotients with derivatives on the RHS of \eqref{3c}, we arrive at \eqref{3a}. 
\end{proof}

As outlined in Section \ref{79}, we proceed to derive a series of integral estimates by multiplying \eqref{3a} by suitable test functions and integrating by parts using the divergence structure proved in Lemma \ref{22}. Recall that for a fixed increment $h>0$, we defined $v(x) = \sum_l \Delta_{ll}^h u(x)$, and that we fixed a ball $B_{2R}\Subset\Omega_h$ and a constant $C_1$ (depending on an upper bound for $\|H\|_{C^0(\Sigma)}$) such that $\Delta u + C_1 \geq 1$ and $|\nabla^2 u| \leq \Delta u  +C_1$ a.e. in $B_{2R}$. The existence of such a constant is guaranteed by the assumption $\lambda(A_H)\in\Gamma_2^+$. We then defined $\tilde{v} = v + C_1$, and for a small parameter $\delta>0$ (that we eventually take to zero) we defined $Q_\delta = \big((\tilde{v}^+)^2 + \delta^2\big)^{1/2}$. For $\rho\in (0,\frac{R}{3}]$, we also fix a cutoff function $\eta\in C_{\mathrm{c}}^\infty(B_{R+2\rho})$ satisfying $0\leq \eta\leq 1$, $\eta\equiv 1$ on $B_{R+\rho}$ and $|\nabla^l\eta| \leq C(n)\rho^{-l}$ for $l=1,2$.

Suppose $u\in W_{\operatorname{loc}}^{2,q+k-1}(\Omega)\cap W_{\operatorname{loc}}^{1,\infty}(\Omega)$ ($q>1$) is a solution to \eqref{7}. Multiplying \eqref{3a} by $\eta Q_\delta^{q-1}$ and integrating over the domain $B_{R+2\rho}$, we see
\begin{align}\label{11'}
\sum_l\int_{B_{R+2\rho}}k\eta Q_\delta^{q-1}f^{k-1}\Delta_{ll}^h f[u] & \leq \int_{B_{R+2\rho}}\eta Q_\delta^{q-1}F^{ij}\nabla_i\nabla_j \tilde{v} - \sum_l\int_{B_{R+2\rho}}\eta Q_\delta^{q-1}F^{ij}\Delta_{ll}^h (H[u])_{ij},
\end{align}
which is just the estimate \eqref{11''} in Section \ref{79}, repeated here for convenience.

We are now in a position to prove our first integral estimate. In what follows, let
\begin{equation*}
J_h^{(s)} \defeq \int_{B_{R+2\rho}}(\tilde{v}^+)^{s} + \int_{B_{R+3\rho}}(\Delta u + C_1)^{s}.
\end{equation*}
Roughly speaking, if $u\in W_{\operatorname{loc}}^{2,s}(\Omega)$ then $J_h^{(s)}$ should be interpreted as a lower order term, and terms bounded by $J_h^{(s)}$ are consequently considered `good terms'.

We will first address the case $f=f(x,z)$ for simplicity and postpone the more general case until Section \ref{AT}. The relevant equation is therefore 
\begin{equation}\label{7''}
\sigma_k^{1/k}\big(A_H[u](x)\big) = f(x,u(x))>0, \quad \lambda(A_H[u](x)) \in\Gamma_k^+  \quad \mathrm{for~a.e.~}x\in\Omega.
\end{equation}
Throughout Section \ref{81}, unless otherwise stated, $C$ will denote a generic positive constant which may vary from line to line, depending only on $n,R,f,H$ and an upper bound for $\| u\|_{W^{1,\infty}(B_{2R})}$. In particular, $C$ is independent of $h$, $q$ and $\rho$, and any norm of $\nabla^2 u$. In addition, we will often use the inequalities $\Delta u + C_1 \geq 1$ and $|\nabla^2 u| \leq \Delta u  +C_1$ without explicit reference.

\begin{lem}\label{P}
	Suppose $f\in C_{\operatorname{loc}}^{1,1}(\Omega\times\mathbb{R})$ is positive, $H\in C_{\operatorname{loc}}^{0,1}(\Omega\times\mathbb{R}\times\mathbb{R}^n\,;\,\operatorname{Sym}_n(\mathbb{R}))$ and $u\in W_{\operatorname{loc}}^{2,q+k-1}(\Omega)\cap W_{\operatorname{loc}}^{1,\infty}(\Omega)$ ($q>1$) is a solution to \eqref{7''}. Then for $R>0$ with $B_{2R}\Subset\Omega$, $\rho\in (0,\frac{R}{3}]$ and $|h|$ sufficiently small, we have
	\begin{align}
	(q-1)\int_{B_{R+2\rho}}\eta(\tilde{v}^+)^{q-2}F^{ij}\nabla_i&\tilde{v}\nabla_j\tilde{v} + \int_{B_{R+2\rho}}\eta(\tilde{v}^+)^{q-1}\nabla_i F[u]^{ij}\nabla_j\tilde{v}  \nonumber \\
	& + \sum_l \int_{B_{R+2\rho}}\eta(\tilde{v}^+)^{q-1} F^{ij}\Delta_{ll}^h (H[u])_{ij} \leq C\rho^{-2} J_h^{(q+k-1)}.\label{AC}
	\end{align}
\end{lem}

\begin{proof}[Proof of Lemma \ref{P}]
Appealing to Lemma \ref{22} with $\phi_j = \eta Q_\delta^{q-1}\nabla_j \tilde{v}$, and noting that
\begin{align*}
\nabla_i\phi_j = Q_\delta^{q-1}\nabla_i\eta\nabla_j \tilde{v} +(q-1)\tilde{v}^+Q_\delta^{q-3}\nabla_i \tilde{v}\nabla_j \tilde{v} + \eta Q_\delta^{q-1}\nabla_i\nabla_j \tilde{v},
\end{align*}
we have
\begin{align}\label{7'}
\int_{B_{R+2\rho}}F^{ij}&\Big( Q_\delta^{q-1}\nabla_i\eta\nabla_j \tilde{v} +(q-1)\tilde{v}^+Q_\delta^{q-3}\nabla_i \tilde{v}\nabla_j \tilde{v} + \eta Q_\delta^{q-1}\nabla_i\nabla_j \tilde{v}\Big) \nonumber \\
& = -\int_{B_{R+2\rho}}\eta Q_\delta^{q-1}\nabla_i F[u]^{ij}\nabla_j \tilde{v}.
\end{align}
Rearranging \eqref{7'} to get the desired integration by parts formula for $\int_{B_{R+2\rho}}\eta Q_\delta^{q-1}F^{ij}\nabla_i\nabla_j \tilde{v}$, and substituting this back into \eqref{11'}, we obtain

\begin{align}\label{12''}
&(q-1)\int_{B_{R+2\rho}}\eta \tilde{v}^+Q_\delta^{q-3}F^{ij}\nabla_i \tilde{v}\nabla_j \tilde{v} + \int_{B_{R+2\rho}}\eta Q_\delta^{q-1}\nabla_i F[u]^{ij}\nabla_j \tilde{v} \nonumber \\
& + \sum_l\int_{B_{R+2\rho}}\eta Q_\delta^{q-1}F^{ij}\Delta_{ll}^h (H[u])_{ij} \leq - \int_{B_{R+2\rho}} Q_\delta^{q-1}F^{ij}\nabla_i\eta\nabla_j \tilde{v}\nonumber \\
& \qquad \qquad \qquad \qquad \qquad \qquad \qquad \quad\,\,\,\,  -  \sum_l\int_{B_{R+2\rho}}k\eta Q_\delta^{q-1}f^{k-1}\Delta_{ll}^h f[u]. 
\end{align}
We now take $\delta\rightarrow 0$ in \eqref{12''}, using Fatou's lemma for the first integral (which is positive) and the dominated convergence theorem elsewhere (which is justified since $q>1$). This yields
\begin{align}\label{12'}
& (q-1)\int_{B_{R+2\rho}}\eta (\tilde{v}^+)^{q-2}F^{ij}\nabla_i \tilde{v}\nabla_j \tilde{v} + \int_{B_{R+2\rho}}\eta (\tilde{v}^+)^{q-1}\nabla_i F[u]^{ij}\nabla_j \tilde{v} \nonumber \\
& + \sum_l\int_{B_{R+2\rho}}\eta (\tilde{v}^+)^{q-1}F^{ij}\Delta_{ll}^h (H[u])_{ij} \leq - \int_{B_{R+2\rho}} (\tilde{v}^+)^{q-1}F^{ij}\nabla_i\eta\nabla_j \tilde{v} \nonumber \\
& \qquad \qquad \qquad \qquad \qquad \qquad \quad \qquad \quad\,\,\,\,\,  - \sum_l\int_{B_{R+2\rho}}k\eta (\tilde{v}^+)^{q-1}f^{k-1}\Delta_{ll}^h f[u].
\end{align}

To conclude the proof of Lemma \ref{P}, we must bound the RHS of \eqref{12'} from above by $C\rho^{-2}J_h^{(q+k-1)}$. We begin with the first integral on the RHS of \eqref{12'}. Appealing again to Lemma \ref{22}, now with $\phi_j = \frac{1}{q}(\tilde{v}^+)^q\nabla_j \eta$ and $\nabla_i\phi_j = (\tilde{v}^+)^{q-1}\nabla_i \tilde{v}\nabla_j \eta + \frac{1}{q}(\tilde{v}^+)^q\nabla_i\nabla_j \eta$, we have
\begin{equation*}
\int_{B_{R+2\rho}}F^{ij}\Big((\tilde{v}^+)^{q-1}\nabla_i \eta\nabla_j \tilde{v} + \frac{1}{q}(\tilde{v}^+)^q\nabla_i\nabla_j \eta\Big)  = -\frac{1}{q}\int_{B_{R+2\rho}}(\tilde{v}^+)^q\nabla_i F[u]^{ij}\nabla_j \eta. 
\end{equation*}
Therefore,
\begin{align}\label{3i}
\bigg|\int_{B_{R+2\rho}} (\tilde{v}^+)^{q-1}F^{ij}\nabla_i\eta\nabla_j \tilde{v}\bigg| & \leq   \bigg|\frac{1}{q}\int_{B_{R+2\rho}}(\tilde{v}^+)^qF^{ij}\nabla_i\nabla_j\eta\bigg| + \bigg|\frac{1}{q}\int_{B_{R+2\rho}}(\tilde{v}^+)^q\nabla_i F[u]^{ij}\nabla_j\eta\bigg|  \nonumber \\
& \leq  \frac{C}{\rho^2}\int_{B_{R+2\rho}}(\tilde{v}^+)^q |F| + \frac{C}{\rho} \int_{B_{R+2\rho}}(\tilde{v}^+)^q \big|\!\operatorname{div}F[u]\big|,
\end{align}
where $F=(F^{ij})$. Recalling $|F|\leq C(\Delta u + C_1)^{k-1}$ and applying H\"older's inequality to the penultimate integral in \eqref{3i},  we see that $\int_{B_{R+2\rho}}(\tilde{v}^+)^q |F| \leq  CJ_h^{(q+k-1)}$. The final integral in \eqref{3i} satisfies the same estimate, since $|\operatorname{div}F[u]|\leq C(\Delta u + C_1)^{k-1}$ by \eqref{3h}.

 It remains to estimate the second term on the RHS of \eqref{12'}. Keeping in mind that $f=f(x,z)\in C_{\operatorname{loc}}^{1,1}(\Omega\times\mathbb{R})$, we apply H\"older's inequality followed by \eqref{3g'} to obtain
 
\begin{align}\label{AF}
\bigg|\sum_l\int_{B_{R+2\rho}}k\eta (\tilde{v}^+)^{q-1}f^{k-1}\Delta_{ll}^h &f[u]\bigg|   \leq C \bigg(\int_{B_{R+2\rho}}(\tilde{v}^+)^q\bigg)^\frac{q-1}{q}\bigg(\int_{B_{R+2\rho}}\bigg|\sum_l\Delta_{ll}^h f[u]\bigg|^{q}\bigg)^{\frac{1}{q}} \nonumber \\
& \leftstackrel{\eqref{3g'}}{\leq}  C\bigg(\int_{B_{R+2\rho}}(\tilde{v}^+)^q\bigg)^\frac{q-1}{q}\bigg(\int_{B_{R+3\rho}}\big|\Delta f[u]\big|^{q}\bigg)^{\frac{1}{q}} \leq CJ_h^{(q)}.
\end{align}
This concludes the proof.
\end{proof}

To clear up notation, we denote the three integrals on the LHS of \eqref{AC} involving higher order terms by
\begin{align*}
\operatorname{(I_1)}_h & \defeq (q-1)\int_{B_{R+2\rho}}\eta (\tilde{v}^+)^{q-2}F^{ij}\nabla_i \tilde{v}\nabla_j \tilde{v}, \nonumber \\
\operatorname{(I_2)}_h & \defeq \int_{B_{R+2\rho}}\eta (\tilde{v}^+)^{q-1}\nabla_i F[u]^{ij}\nabla_j \tilde{v}\,\, \text{ and} \nonumber \\
\operatorname{(I_3)}_h & \defeq \sum_l\int_{B_{R+2\rho}}\eta (\tilde{v}^+)^{q-1}F^{ij}\Delta_{ll}^h (H[u])_{ij}. 
\end{align*}

The terms $\operatorname{(I_1)}_h$, $\operatorname{(I_2)}_h$ and $\operatorname{(I_3)}_h$ will be considered in turn. In Section \ref{A}, we prove an estimate for $\operatorname{(I_1)}_h$. In Section \ref{AD}, we estimate $\operatorname{(I_2)}_h$ in the case that $H$ is a multiple of the identity, and in Section \ref{C} we estimate $\operatorname{(I_2)}_h$ for general $H$ when $k=2$. The estimate for $\operatorname{(I_3)}_h$ in the general case is slightly involved, so for illustrative purposes we first address the simpler case when $H(x,z,\xi)=H_1(x,z)|\xi|^2I$ with $H_1\geq 0$, which includes the $\sigma_k$-Yamabe equation in the positive case. This is done in Section \ref{3n}. The estimate for $\operatorname{(I_3)}_h$ in the general case is proved in Section \ref{D}. In the process, we will prove the cancellation phenomenon between $\operatorname{(I_2)}_h$ and $\operatorname{(I_3)}_h$ alluded to earlier -- see Corollaries \ref{J}, \ref{M} and \ref{N}. 

\subsection{A pointwise lower bound for $F[u]^{ij}\nabla_i\tilde{v}\nabla_j\tilde{v}$}\label{A}

The term $F^{ij}\nabla_i\tilde{v}\nabla_j \tilde{v}$ in $\operatorname{(I_1)}_h$ can be bounded in the same way as in \cite{Urb00} (see equation (3.6) therein). We reproduce the argument here for the reader's convenience. 

\begin{lem}\label{21'}
	Suppose $f\in C^0(\Omega\times\mathbb{R})$ is positive, $H\in C^0(\Omega\times\mathbb{R}\times\mathbb{R}^n\,;\,\operatorname{Sym}_n(\mathbb{R}))$ and $u$ is a solution to \eqref{7''}. Then for $q>0$,	\begin{equation}\label{21''}
	(v^+)^{q-2}F^{ij}\nabla_i\tilde{v}\nabla_j \tilde{v} \geq \frac{4f^k}{q^2}\frac{\big|\nabla\big((\tilde{v}^+)^{q/2}\big)\big|^2}{\Delta u - \operatorname{tr}(H)}\quad \mathrm{a.e.~in~}\Omega_h. 
	\end{equation}
	In particular, for $R>0$ with $B_{2R}\Subset\Omega$, $\rho\in (0,\frac{R}{3}]$, $q>1$ and $|h|$ sufficiently small, we have
	\begin{equation}\label{21}
	\operatorname{(I_1)}_h \geq \frac{q-1}{Cq^2}\int_{B_{R+\rho}}f^k\frac{\big|\nabla\big((\tilde{v}^+)^{q/2}\big)\big|^2}{\Delta u - \operatorname{tr}(H)}. 
	\end{equation}
\end{lem}

\begin{proof}
	Denote by $\mathcal{M}_{k}^+\subset\operatorname{Sym}_n(\mathbb{R})$ the set of symmetric matrices $M$ with $\lambda(M)\in\Gamma_{k}^+$. For $1\leq l\leq n$, denote by $F_{(l)}^{ij}(A)$ the matrix with entries $\partial \sigma_l(A)/\partial A_{ij}$. Using the concavity of $\sigma_k(A)/\sigma_{k-1}(A)$ on $\mathcal{M}_k^+$, we have
	\begin{equation}\label{32}
	\frac{F_{(k)}^{ij}(A)}{\sigma_k(A)} \geq \frac{F_{(k-1)}^{ij}(A)}{\sigma_{k-1}(A)}\quad \text{for all }A\in\mathcal{M}_k^+
	\end{equation}
	(see e.g. \cite{Urb00, LT94}). Applying \eqref{32} inductively, it follows that
	\begin{equation}\label{20}
	\frac{F_{(k)}^{ij}(A)}{\sigma_k(A)} \geq \cdots \geq \frac{F_{(1)}^{ij}(A)}{\sigma_1(A)} = \frac{\delta^{ij}}{\operatorname{tr}(A)}\quad \text{for all }A\in\mathcal{M}_k^+.
	\end{equation}
	Taking $A=A_H[u]$ in \eqref{20}, where $u$ is a solution to \eqref{7''}, we obtain
	\begin{equation*}
	\frac{F[u]^{ij}(x)}{f^k[u](x)} \geq \frac{\delta^{ij}}{\Delta u(x) - \operatorname{tr}(H[u](x))} \quad \text{for a.e. }x\in\Omega,
	\end{equation*}
	from which \eqref{21''} is readily seen. The estimate \eqref{21} then follows from properties of $\eta$. 
\end{proof}

\subsection{Integral estimates for $\nabla_i F[u]^{ij}\nabla_j\tilde{v}$}\label{B}

In this section we obtain estimates for the term $\operatorname{(I_2)}_h = \int_{B_{R+2\rho}}\eta(\tilde{v}^+)^{q-1}\nabla_i F[u]^{ij}\nabla_j \tilde{v}$. The case in which $H$ is a multiple of the identity matrix will be dealt with first, in Section \ref{AD}. The case for general $H$ when $k=2$ will then be addressed in Section \ref{C}.

\subsubsection{The case $H=H_2(x,z,\xi)I$}\label{AD}

In what follows we denote $\operatorname{tr}(F) = \sum_i F^{ii}$. We prove the following two lemmas which address the case that $H$ is a multiple of the identity:

\begin{lem}\label{83}
	Suppose $f\in C^0(\Omega\times\mathbb{R})$ is positive, $H\in C_{\operatorname{loc}}^{1,1}(\Omega\times\mathbb{R}\times\mathbb{R}^n\,;\,\operatorname{Sym}_n(\mathbb{R}))$ with $H(x,z,\xi)=H_1(x,z)|\xi|^2I$, and that $u\in W_{\operatorname{loc}}^{2,q+k-1}(\Omega)\cap W_{\operatorname{loc}}^{1,\infty}(\Omega)$ ($q>1$) is a solution to \eqref{7''}. Then for $R>0$ with $B_{2R}\Subset\Omega$, $\rho\in (0,\frac{R}{3}]$ and $|h|$ sufficiently small, we have
	\begin{align}\label{45}
	\operatorname{(I_2)}_h & \geq  -\int_{B_{R+2\rho}}\eta(\tilde{v}^+)^{q-1}\operatorname{tr}(F) \frac{\partial (H_1|\xi|^2)}{\partial \xi_a}[u]\nabla_a \tilde{v} - C\rho^{-1}J_h^{(q+k-1)}.
	\end{align}
\end{lem}

\begin{lem}\label{83'}
	Suppose $H\in C_{\operatorname{loc}}^{1,1}(\Omega\times\mathbb{R}\times\mathbb{R}^n\,;\,\operatorname{Sym}_n(\mathbb{R}))$ with $H(x,z,\xi) = H_2(x,z,\xi)I$, and that $u\in W_{\operatorname{loc}}^{2,q+k}(\Omega)\cap W_{\operatorname{loc}}^{1,\infty}(\Omega)$ ($q>1$). Then for $R>0$ with $B_{2R}\Subset\Omega$, $\rho\in (0,\frac{R}{3}]$ and $|h|$ sufficiently small, we have
	\begin{align}\label{45'}
	\operatorname{(I_2)}_h & \geq -\int_{B_{R+2\rho}}\eta(\tilde{v}^+)^{q-1}\operatorname{tr}(F) \frac{\partial H_2}{\partial \xi_a}[u]\nabla_a \tilde{v} - C\rho^{-1}J_h^{(q+k)}.
	\end{align}
\end{lem}

\begin{rmk}
	Note that in Lemma \ref{83'}, we do not assume that $u$ solves \eqref{7''}. In contrast, the fact that Lemma \ref{83} holds under a weaker integrability assumption uses both the fact that $u$ solves \eqref{7''} \textit{and} that $H_2$ depends quadratically on $\nabla u$.
\end{rmk}

\begin{rmk}\label{14}
	The first term on the RHS of \eqref{45} and \eqref{45'} will later be shown to cancel with a term arising from our estimate for $\operatorname{(I_3)}_h$. 
\end{rmk}

\begin{proof}[Proof of Lemmas \ref{83} and \ref{83'}]
The proof consists of three steps. In Step 1, we prove a preliminary estimate assuming only $u\in W_{\operatorname{loc}}^{2,q+k-1}(\Omega)\cap W_{\operatorname{loc}}^{1,\infty}(\Omega)$ and $H=H_2(x,z,\xi)I$, but we do not assume at this point that $u$ necessarily solves \eqref{7''}. Only in Steps 2 and 3 will we appeal to the specific hypotheses of Lemmas \ref{83} and \ref{83'}.

Our starting point is the following expression for $\operatorname{(I_2)}_h$, which follows from \eqref{15'}: 

\begin{equation*}
\operatorname{(I_2)}_h = -(n-k+1)\int_{B_{R+2\rho}}\eta(\tilde{v}^+)^{q-1}\nabla_j (H_2[u])\, T_{k-2}(A_H)^{ij}\nabla_i \tilde{v}. 
\end{equation*}

\noindent \textit{\textbf{Step 1:}} In this step, we show that for every $u\in W_{\operatorname{loc}}^{2,q+k-1}(\Omega)\cap W_{\operatorname{loc}}^{1,\infty}(\Omega)$,
\begin{align}\label{3j}
\operatorname{(I_2)}_h & \geq -\int_{B_{R+2\rho}}\eta(\tilde{v}^+)^{q-1}\operatorname{tr}(F) \frac{\partial H_2}{\partial \xi_a}[u]\nabla_a\tilde{v} -\frac{n-k+1}{q}\int_{B_{R+2\rho}}\eta(\tilde{v}^+)^qF^i_a\frac{\partial^2 H_2}{\partial \xi_a\partial \xi_b}[u](A_H)_{ib}  \nonumber \\
& \quad  - C\rho^{-1}J_h^{(q+k-1)}.
\end{align}
Note that the first integral on the RHS of \eqref{3j} is the desired term seen in \eqref{45} and \eqref{45'}.  

First observe that by the chain rule,
\begin{align}\label{30}
\operatorname{(I_2)}_h & = -(n-k+1)\int_{B_{R+2\rho}}\eta(\tilde{v}^+)^{q-1}\frac{\partial H_2}{\partial x^j}[u]\,T_{k-2}(A_H)^{ij}\nabla_i \tilde{v} \nonumber \\
& \quad - (n-k+1)\int_{B_{R+2\rho}}\eta(\tilde{v}^+)^{q-1}\frac{\partial H_2}{\partial z}[u]\,T_{k-2}(A_H)^{ij}\nabla_j u\nabla_i \tilde{v} \nonumber \\
& \quad - (n-k+1)\int_{B_{R+2\rho}}\eta(\tilde{v}^+)^{q-1}\frac{\partial H_2}{\partial \xi_a}[u]\,T_{k-2}(A_H)^{ij}\nabla_j\nabla_a u\, \nabla_i \tilde{v}. 
\end{align}
Denote the top two lines of the RHS of \eqref{30} collectively by $L_1$, and the bottom line  by $L_2$. Recalling that $\nabla_j\nabla_a u  =  H_2\delta_{ja} + (A_H)_{ja}$ and, in view of \eqref{28} and \eqref{1''}, that
 \begin{equation}\label{30''}
 \big(T_{k-2}(A_H) A_H\big)_{ia}  =  - F_{ia} + \frac{1}{n-k+1}\operatorname{tr}(F)\delta_{ia},
 \end{equation}
 we have
\begin{align*}
L_2 & = - (n-k+1)\int_{B_{R+2\rho}}\eta(\tilde{v}^+)^{q-1}\frac{\partial H_2}{\partial \xi_a}[u]\,T_{k-2}(A_H)_{ia}H_2\nabla^i \tilde{v}\nonumber \\
& \qquad \, -(n-k+1)\int_{B_{R+2\rho}}\eta(\tilde{v}^+)^{q-1}\frac{\partial H_2}{\partial \xi_a}[u]\,\big(T_{k-2}(A_H) A_H\big)_{ia}\nabla^i \tilde{v}   \nonumber \\
& \leftstackrel{\eqref{30''}}{=}  - (n-k+1)\int_{B_{R+2\rho}}\eta(\tilde{v}^+)^{q-1}\frac{\partial H_2}{\partial \xi_a}[u]\,T_{k-2}(A_H)_{ia}H_2\nabla^i \tilde{v}\nonumber \\
& \qquad \, + (n-k+1)\int_{B_{R+2\rho}}\eta(\tilde{v}^+)^{q-1}\frac{\partial H_2}{\partial \xi_a}[u]\,F_{ia}\nabla^i \tilde{v}-\int_{B_{R+2\rho}}\eta(\tilde{v}^+)^{q-1}\operatorname{tr}(F)\frac{\partial H_2}{\partial \xi_a}[u]\,\nabla_a \tilde{v}. 
\end{align*}
Substituting this identity for $L_2$ into \eqref{30} yields
\begin{align}\label{4'}
\operatorname{(I_2)}_h & = L_1 - (n-k+1)\int_{B_{R+2\rho}}\eta(\tilde{v}^+)^{q-1}\frac{\partial H_2}{\partial \xi_a}[u]\,T_{k-2}(A_H)_{ia}H_2\nabla^i \tilde{v} \nonumber \\
& \quad + (n-k+1)\int_{B_{R+2\rho}}\eta(\tilde{v}^+)^{q-1}\frac{\partial H_2}{\partial \xi_a}[u]\,F_{ia}\nabla^i \tilde{v} -\int_{B_{R+2\rho}}\eta(\tilde{v}^+)^{q-1}\operatorname{tr}(F)\frac{\partial H_2}{\partial \xi_a}[u]\,\nabla_a \tilde{v}.
\end{align}
We claim that the terms on the top line of the RHS of \eqref{4'} are bounded from below by $- C\rho^{-1}J_h^{(q+k-1)}$. Indeed, as $T_{k-2}(A_H)^{ij} = \partial\sigma_{k-1}(A_H)/\partial A_{ij}$, by Lemma \ref{22} we have $|\nabla_i T_{k-2}(A_H[u])^{ij}| \leq C(\Delta u + C_1)^{k-2}$. It is also clear that $|T_{k-2}(A_H)^{ij}| \leq C(\Delta u + C_1)^{k-2}$. Thus, after integrating by parts using Lemma \ref{22} and applying H\"older's inequality, the lower bound for these terms follows. 

To estimate the penultimate integral in \eqref{4'}, we integrate by parts using Lemma \ref{22} and apply the identity 
\begin{equation*}
\nabla_i\bigg(\frac{\partial H_2}{\partial \xi_a}[u](x)\bigg) = \bigg(\frac{\partial^2 H_2}{\partial \xi_a\partial \xi_b}[u](x)\bigg)\big((A_H)_{ib} + H_{ib}\big) + \bigg(\frac{\partial^2 H_2}{\partial z\partial \xi_a}[u](x)\bigg)\nabla_i u(x) +   \frac{\partial^2 H_2}{\partial x^i\partial \xi_a} [u](x).
\end{equation*}
After an application of H\"older's inequality, this gives
\begin{align*}
\int_{B_{R+2\rho}}\eta(\tilde{v}^+)^{q-1}\frac{\partial H_2}{\partial \xi_a}[u]\,F_{ia}\nabla^i \tilde{v} & \geq  -\frac{1}{q}\int_{B_{R+2\rho}}\eta(\tilde{v}^+)^qF^i_a\frac{\partial^2 H_2}{\partial \xi_a\partial \xi_b}[u]\,(A_H)_{ib} - C\rho^{-1}J_h^{(q+k-1)},
\end{align*}
from which \eqref{3j} follows. \newline

\noindent \textit{\textbf{Step 2:}} In this step we prove Lemma \ref{83'}. Indeed, for $u\in W_{\operatorname{loc}}^{2,q+k}(\Omega)\cap W_{\operatorname{loc}}^{1,\infty}(\Omega)$ (not necessarily solving \eqref{7''}) we have the estimate 
\begin{align*}
-\frac{n-k+1}{q}\int_{B_{R+2\rho}}\eta(\tilde{v}^+)^qF^i_a\frac{\partial^2 H_2}{\partial \xi_a\partial \xi_b}[u]\,(A_H)_{ib} & \geq -C\int_{B_{R+2\rho}}(\tilde{v}^+)^q|F||A_H| \geq - CJ_h^{(q+k)},
\end{align*}
where $F=(F^{ij})$ and the last inequality follows once again from the estimate $|F|\leq C(\Delta u + C_1)^{k-1}$ and H\"older's inequality. Substituting this into \eqref{3j} then yields the desired estimate \eqref{45'}. \newline

\noindent \textit{\textbf{Step 3:}} In this step we prove Lemma \ref{83}. Since we assume in this case that $H_2(x,z,\xi) = H_1(x,z)|\xi|^2$ and that $u$ solves \eqref{7''}, rather than estimating as in Step 2 we observe
\begin{equation}\label{44}
F^i_a\frac{\partial^2 H_2}{\partial \xi_a \partial \xi_b}[u](A_H)_{ib} =2H_1F^i_a\delta^{ab}(A_H)_{ib} = 2H_1F^i_a (A_H)^a_i = 2H_1k\sigma_k(A_H) = 2H_1kf^k. 
\end{equation}
Substituting \eqref{44} into the second integral in \eqref{3j}, we arrive at \eqref{45}.
\end{proof}

\subsubsection{The case $k=2$ for general $H$}\label{C}

In this section we obtain an estimate in the case $k=2$ analogous to \eqref{45} and \eqref{45'}. We do not assume that $H$ is a multiple of the identity and, as in Lemma \ref{83'}, we do not assume that $u$ solves \eqref{7''}:

\begin{lem}\label{84}
	Suppose $H\in C_{\operatorname{loc}}^{1,1}(\Omega\times\mathbb{R}\times\mathbb{R}^n\,;\,\operatorname{Sym}_n(\mathbb{R}))$, $k=2$ and $u\in W_{\operatorname{loc}}^{2,q+2}(\Omega)\cap W_{\operatorname{loc}}^{1,\infty}(\Omega)$ ($q>1$). Then for  $R>0$ with $B_{2R}\Subset\Omega$, $\rho\in(0,\frac{R}{3}]$ and $|h|$ sufficiently small, we have
	\begin{align}\label{46}
	\operatorname{(I_2)}_h & \geq \int_{B_{R+2\rho}}\eta(\tilde{v}^+)^{q-1}\frac{\partial H^{ij}}{\partial \xi_a}[u]\,\nabla_i\nabla_a u\,\nabla_j \tilde{v} - \int_{B_{R+2\rho}}\eta(\tilde{v}^+)^{q-1}\frac{\partial \operatorname{tr}(H)}{\partial \xi_a}[u]\,\operatorname{tr}(A_H)\nabla_a \tilde{v} \nonumber \\
	& \quad  - C\rho^{-1}J_h^{(q+2)}.
	\end{align}
\end{lem}

\begin{rmk}
	The first two terms on the RHS of \eqref{46} will later be shown to cancel with a term arising from our estimate for $\operatorname{(I_3)}_h$ (cf. Remark \ref{14}).
\end{rmk}

\begin{proof}[Proof of Lemma \ref{84}]
As $k=2$, we have $\nabla_i F[u]^{ij} = \nabla_i H[u]^{ij} - \nabla^j \operatorname{tr}(H[u])$ (by \eqref{15}) and $\nabla^j\nabla^a u = \operatorname{tr}(A_H)\delta^{ja} - F^{ja} - H^{ja}$. It follows that
\begin{align}
&\operatorname{(I_2)}_h  = \int_{B_{R+2\rho}}\eta(\tilde{v}^+)^{q-1}\big(\nabla_i H[u]^{ij}-\nabla^j\operatorname{tr}(H[u]) \big)\nabla_j \tilde{v} \nonumber \\
& = \int_{B_{R+2\rho}}\eta(\tilde{v}^+)^{q-1}\bigg(\frac{\partial H^{ij}}{\partial \xi_a}[u]\,\nabla_i\nabla_a u - \frac{\partial \operatorname{tr}(H)}{\partial \xi^a}[u]\,\nabla^j\nabla^a u\bigg)\nabla_j\tilde{v}  \nonumber \\
& \qquad + \int_{B_{R+2\rho}}\eta(\tilde{v}^+)^{q-1}\bigg(\frac{\partial H^{ij}}{\partial x^i}[u] + \frac{\partial H^{ij}}{\partial z}[u]\,\nabla_i u - \frac{\partial \operatorname{tr}(H)}{\partial x_j}[u] - \frac{\partial \operatorname{tr}(H)}{\partial z}[u]\,\nabla^j u\bigg)\nabla_j\tilde{v}\nonumber 
\end{align}
\begin{align}\label{3k}
& = \int_{B_{R+2\rho}}\eta(\tilde{v}^+)^{q-1}\bigg(\frac{\partial H^{ij}}{\partial \xi_a}[u]\,\nabla_i\nabla_a u - \frac{\partial \operatorname{tr}(H)}{\partial \xi^a}[u]\,\operatorname{tr}(A_H)\delta^{ja}\bigg)\nabla_j\tilde{v} \nonumber \\
& \qquad + \int_{B_{R+2\rho}}\eta(\tilde{v}^+)^{q-1}\bigg(\frac{\partial H^{ij}}{\partial x^i}[u] + \frac{\partial H^{ij}}{\partial z}[u]\,\nabla_i u - \frac{\partial \operatorname{tr}(H)}{\partial x_j}[u] - \frac{\partial \operatorname{tr}(H)}{\partial z}[u]\,\nabla^j u \nonumber \\
& \qquad \qquad \qquad \qquad \qquad  + \frac{\partial\operatorname{tr}(H)}{\partial \xi^a}[u](F^{ja}+H^{ja}) \bigg)\nabla_j\tilde{v}.
\end{align}

The integral on the last two lines of \eqref{3k} can be bounded from below by $- C\rho^{-1}J_h^{(q+2)}$ in exactly the same way as in the proof of Lemmas \ref{83} and \ref{83'}: we integrate by parts using Lemma \ref{22}, estimate the relevant quantities in terms of $\Delta u + C_1$ and apply H\"older's inequality. The estimate \eqref{46} then follows. 
\end{proof} 

\subsection{Integral estimates for $F[u]^{ij}\Delta_{ll}^h H[u]_{ij}$}\label{G}

In this section we obtain estimates for the quantity $\operatorname{(I_3)}_h = \sum_l\int_{B_{R+2\rho}}\eta(\tilde{v}^+)^{q-1}F^{ij}\Delta_{ll}^h (H[u])_{ij}$. More precisely, we will prove the following lemma: 

\begin{lem}\label{71}
	Suppose $H\in C_{\operatorname{loc}}^{1,1}(\Omega\times\mathbb{R}\times\mathbb{R}^n\,;\,\operatorname{Sym}_n(\mathbb{R}))$, $R>0$ is such that $B_{2R}\Subset\Omega$ and $\rho\in(0,\frac{R}{3}]$. 
	\begin{enumerate}
		\item[a)] If $u\in W_{\operatorname{loc}}^{2,q+k}(\Omega)\cap W_{\operatorname{loc}}^{1,\infty}(\Omega)$ ($q>1$), then for $|h|$ sufficiently small, we have
		\begin{equation}\label{37'}
		\operatorname{(I_3)}_h \geq \int_{B_{R+2\rho}}\eta (\tilde{v}^+)^{q-1}F^{ij}\frac{\partial H_{ij}}{\partial \xi_a}[u]\,\nabla_a \tilde{v} - CJ_h^{(q+k)}.
		\end{equation} 
		\item[b)] If $u\in W_{\operatorname{loc}}^{2,q+k-1}(\Omega)\cap W_{\operatorname{loc}}^{1,\infty}(\Omega)$ ($q>1$) and $H(x,z,\xi)=H_1(x,z)|\xi|^2 I$ with $H_1 \geq 0$, then for $|h|$ sufficiently small, we have
		\begin{equation}\label{37}
		\operatorname{(I_3)}_h \geq \int_{B_{R+2\rho}}\eta (\tilde{v}^+)^{q-1}F^{ij}\frac{\partial H_{ij}}{\partial \xi_a}[u]\,\nabla_a \tilde{v} - CJ_h^{(q+k-1)}.
		\end{equation} 
	\end{enumerate}
\end{lem}
\begin{rmk}
	Neither estimate in Lemma \ref{71} requires $u$ to be a solution to \eqref{7''}. 
\end{rmk}

Before proving Lemma \ref{71} we first discuss its consequences, namely the resulting cancellations between $\operatorname{(I_2)}_h$ and $\operatorname{(I_3)}_h$. First consider the case $H=H_1(x,z)|\xi|^2I$ with $H_1\geq 0$:

\begin{cor}\label{J}
	Suppose $f\in C_{\operatorname{loc}}^{1,1}(\Omega\times\mathbb{R})$ is positive, $H\in C_{\operatorname{loc}}^{1,1}(\Omega\times\mathbb{R}\times\mathbb{R}^n\,;\,\operatorname{Sym}_n(\mathbb{R}))$ with $H = H_1(x,z)|\xi|^2I$ and $H_1 \geq 0$, and that $u\in W_{\operatorname{loc}}^{2,q+k-1}(\Omega)\cap W_{\operatorname{loc}}^{1,\infty}(\Omega)$ ($q>1$) is a solution to \eqref{7''}. Then for $R>0$ with $B_{2R}\Subset\Omega$, $\rho\in(0,\frac{R}{3}]$ and $|h|$ sufficiently small, we have
	\begin{equation}\label{AG}
	\operatorname{(I_2)}_h + \operatorname{(I_3)}_h \geq - C\rho^{-1}J_h^{(q+k-1)}.
	\end{equation}
	In particular, 
	\begin{equation}\label{r''}
	\begin{split}
	\frac{q-1}{q^2}  \int_{B_{R+\rho}} {f}^k \frac{\big|\nabla\big((\tilde{v}^+)^{q/2}\big)\big|^2}{\Delta u - \operatorname{tr}(H)}  \leq C\rho^{-2}J_h^{(q+k-1)}.
	\end{split}
	\end{equation}
\end{cor}
\begin{proof}
	The estimate \eqref{AG} follows from combining the estimates \eqref{45} and \eqref{37}. The estimate \eqref{r''} is then obtained by substituting \eqref{21} and \eqref{AG} into \eqref{AC}. 
\end{proof}

Similarly, we obtain the following in the case that $H=H_2(x,z,\xi)I$:

\begin{cor}\label{M}
	Suppose $H\in C_{\operatorname{loc}}^{1,1}(\Omega\times\mathbb{R}\times\mathbb{R}^n\,;\,\operatorname{Sym}_n(\mathbb{R}))$ with $H=H_2(x,z,\xi)I$, and $u\in W_{\operatorname{loc}}^{2,q+k}(\Omega)\cap W_{\operatorname{loc}}^{1,\infty}(\Omega)$ ($q>1$). Then for $R>0$ with $B_{2R}\Subset\Omega$, $\rho\in(0,\frac{R}{3}]$ and $|h|$ sufficiently small, we have
	\begin{equation}\label{K}
	\operatorname{(I_2)}_h + \operatorname{(I_3)}_h \geq - C\rho^{-1}J_h^{(q+k)}.
	\end{equation}
	If, in addition, $u$ solves \eqref{7''} for some positive $f\in C_{\operatorname{loc}}^{1,1}(\Omega\times\mathbb{R})$, then 
	\begin{equation}\label{s'''}
	\begin{split}
	\frac{q-1}{q^2}  \int_{B_{R+\rho}} {f}^k \frac{\big|\nabla\big((\tilde{v}^+)^{q/2}\big)\big|^2}{\Delta u - \operatorname{tr}(H)}  \leq   C\rho^{-2}J_h^{(q+k)}.
	\end{split}
	\end{equation}
\end{cor}

\begin{proof}
	The estimate \eqref{K} follows from combining the estimates \eqref{45'} and \eqref{37'}. The estimate \eqref{s'''} is then obtained by substituting \eqref{21} and \eqref{K} into \eqref{AC}. 
\end{proof}

A similar cancellation also holds in the setting of Theorem \ref{56}, although this requires a little more work:

\begin{cor}\label{N}
	Suppose $H\in C_{\operatorname{loc}}^{1,1}(\Omega\times\mathbb{R}\times\mathbb{R}^n\,;\,\operatorname{Sym}_n(\mathbb{R}))$, $k=2$ and $u\in W_{\operatorname{loc}}^{2,q+2}(\Omega)\cap W_{\operatorname{loc}}^{1,\infty}(\Omega)$ ($q>1$). Then for $|h|$ sufficiently small, we have
	\begin{equation}\label{L}
	\operatorname{(I_2)}_h + \operatorname{(I_3)}_h \geq - C\rho^{-2}J_h^{(q+2)}.
	\end{equation}
	If, in addition, $u$ solves \eqref{7''} for some positive $f\in C_{\operatorname{loc}}^{1,1}(\Omega\times\mathbb{R})$, then 
	\begin{equation}\label{s''}
	\begin{split}
	\frac{q-1}{q^2}  \int_{B_{R+\rho}} {f}^2 \frac{\big|\nabla\big((\tilde{v}^+)^{q/2}\big)\big|^2}{\Delta u - \operatorname{tr}(H)}  \leq  C\rho^{-1}J_h^{(q+2)}.
	\end{split}
	\end{equation}
\end{cor}
\begin{proof}
	The estimate \eqref{s''} will immediately follow once \eqref{L} is established, by substituting \eqref{21} and \eqref{L} into \eqref{AC}. 
	
	Taking $k=2$ in Lemma \ref{71} a) and using $F^{ij} = \operatorname{tr}(A_H)\delta^{ij} - \nabla^i\nabla^j u - H^{ij}$, we see
\begin{align}\label{76}
\operatorname{(I_3)}_h & \geq \int_{B_{R+2\rho}}\eta(\tilde{v}^+)^{q-1}\operatorname{tr}(A_H)\frac{\partial \operatorname{tr}(H)}{\partial \xi_a}[u]\,\nabla_a\tilde{v} - \int_{B_{R+2\rho}}\eta(\tilde{v}^+)^{q-1}\nabla^i\nabla^j u\frac{\partial H_{ij}}{\partial \xi_a}[u]\,\nabla_a\tilde{v} \nonumber \\
& \quad - \int_{B_{R+2\rho}}\eta(\tilde{v}^+)^{q-1}H^{ij}\frac{\partial H_{ij}}{\partial \xi_a}[u]\,\nabla_a\tilde{v} - CJ_h^{(q+2)}.
\end{align}
Now, the first term on the RHS of \eqref{76} cancels with the second term on the RHS of \eqref{46}, and the first term on the last line of \eqref{76} can be estimated by $-C\rho^{-1}J_h^{(q+2)}$, after integrating by parts and applying H\"older's inequality. Therefore, combining \eqref{46} and \eqref{76}, we obtain
\begin{align}
\operatorname{(I_2)}_h + \operatorname{(I_3)}_h  \geq  \frac{1}{q}\int_{B_{R+2\rho}}\eta\frac{\partial H_{ij}}{\partial \xi_a}[u]\,\bigg(\nabla^i\nabla_a u\,\nabla^j (\tilde{v}^+)^q - \nabla^i\nabla^j u\,\nabla_a (\tilde{v}^+)^q \bigg) - C\rho^{-1}J_h^{(q+2)}.\nonumber
\end{align}

Now, if $u$ were to have enough regularity, we could integrate by parts here, observe that the third derivatives of $u$ cancel, and obtain \eqref{L} by estimating the remaining terms in the usual way. To circumvent the lack of regularity, we instead apply the following lemma:

\begin{lem}\label{40}
	Let $U\subset\mathbb{R}^n$ be a smooth bounded domain and let $B\in L^\infty(U;\mathbb{R}^{n\times n})$ be an antisymmetric matrix with $\operatorname{supp}(B)\Subset U$. For $1\leq p<\infty$ and $p'\defeq\frac{p}{p-1}$, consider the bilinear form $\mathcal{B}:W^{1,p}(U)\times W^{1,p'}(U)\rightarrow \mathbb{R}$ given by
	\begin{equation}\label{47}
	\mathcal{B}(g,h) = \int_U B_j^a\,\nabla_a g\,\nabla^j h.
	\end{equation}
	If $\operatorname{div}B\in L^q(U;\mathbb{R}^n)$ with $\frac{1}{p} + \frac{1}{q} = 1 - \frac{1}{r}$ for some $1\leq q,r\leq \infty$, then we have the estimate
	\begin{equation}\label{43}
	|\mathcal{B}(g,h)| \leq \int_U |\operatorname{div}B||\nabla g||h|
	\end{equation}
	for all $g\in W^{1,p}(U)$ and $h \in W^{1,p'}(U)\cap L^r(U)$. 
\end{lem}
Before proving Lemma \ref{40}, we use it to complete the proof of \eqref{L}: for each $i\in\{1,\dots,n\}$, taking $B^a_j = \eta \frac{\partial H_{ij}}{\partial \xi_a}[u] - \eta\frac{\partial H_i^a}{\partial \xi^j}[u]$, $g= \nabla_i u$ and $h = (\tilde{v}^+)^q$ in Lemma \ref{40} we obtain 

\begin{align}
\int_{B_{R+2\rho}}\eta\frac{\partial H_{ij}}{\partial \xi_a}[u]\,\bigg(\nabla^i \nabla_a u \nabla^j(\tilde{v}^+)^q-\nabla^i\nabla^j u\nabla_a(\tilde{v}^+)^q\bigg) &\stackrel{\eqref{43}}{\leq}  C\rho^{-1}\int_{B_{R+2\rho}}(\Delta u +C_1)^{2}(\tilde{v}^+)^q \nonumber \\
&  \,\,\, \leq  \,\, \,C\rho^{-1}J_h^{(q+2)}. \nonumber
\end{align}

It remains to prove Lemma \ref{40}. By a standard approximation argument, it suffices to prove \eqref{43} for $g,h\in C^\infty(U)$. We are then justified in integrating by parts in \eqref{47}, giving
\begin{align*}
|\mathcal{B}(g,h)|  = \bigg|\int_U \bigg(\nabla^j B_j^a\, \nabla_a g + \underbrace{B^a_j\,\nabla_a\nabla^jg}_{=0}\bigg)h\bigg|  \leq \int_U |\operatorname{div}B||\nabla g||h|, 
\end{align*}
where we have used antisymmetry of $B$ to assert that $B^a_j\,\nabla_a\nabla^jg = 0$. 
\end{proof}

\subsubsection{Proof of Lemma \ref{71} b)}\label{3n}

We now turn our attention back to the proof of Lemma \ref{71}. Whilst the two estimates \eqref{37'} and \eqref{37} can be dealt with simultaneously (see the proof of Lemma \ref{71} in Section \ref{D}), for illustrative purposes we first provide a more direct proof of \eqref{37}, which includes the $\sigma_k$-Yamabe equation in the positive case. Indeed, when $H=H_1(x,z)|\xi|^2I$ we are able to calculate $\Delta_{ll}^h (H[u])_{ij}$ explicitly by deriving the following discrete version of the Bochner identity, avoiding the more involved estimates required for the general case. In what follows, we denote
\begin{equation*}
u_l^h(x)\defeq u(x+he_l).
\end{equation*}

\begin{lem}[Discrete Bochner identity]\label{H'}
	Suppose $H_1\in C^0(\Omega\times\mathbb{R})$ and $l\in\{1,\dots,n\}$. Then
	\begin{align}\label{q}
	\Delta_{ll}^h\big(H_1[u]|\nabla u|^2\big) & = 2H_1\nabla^i u\nabla_i\Delta_{ll}^h u + (H_1[u])^{-h}_l\big|\nabla\nabla_l^{-h} u\big|^2 + (H_1[u])^h_l\big|\nabla\nabla_l^h u\big|^2 \nonumber \\
	& \quad + \nabla_l^{-h}\nabla_i u \nabla^i u \nabla_l^{-h}H_1[u] + \nabla_l^h\nabla^i u \nabla_i u \nabla_l^h H_1[u] \nonumber \\
	& \quad + \nabla_l^h\Big(\nabla_i u(\nabla^i u)_l^{-h}\nabla_l^{-h}H_1[u]\Big). 
	\end{align}
\end{lem}

Assuming the validity of Lemma \ref{H'}, the proof of \eqref{37} in Lemma \ref{71} b) is then straightforward:

\begin{proof}[Proof of Lemma \ref{71} b)]
	Substituting the discrete Bochner identity \eqref{q} into the definition of $\operatorname{(I_3)}_h$ and dropping the two positive terms, we obtain
	\begin{align}\label{3l}
	\operatorname{(I_3)}_h & \geq 2\int_{B_{R+2\rho}}\eta(\tilde{v}^+)^{q-1}\operatorname{tr}(F)H_1\nabla^i u \nabla_i \tilde{v}  \nonumber \\
	& \quad + \sum_l \int_{B_{R+2\rho}}\eta(\tilde{v}^+)^{q-1}\operatorname{tr}(F)\bigg(\nabla_l^{-h}\nabla_i u \nabla^i u \nabla_l^{-h}H_1[u] + \nabla_l^h\nabla^i u \nabla_i u \nabla_l^h H_1[u] \nonumber \\
	& \qquad \qquad \qquad \qquad \qquad \qquad \qquad +  \nabla_l^h\Big(\nabla_i u(\nabla^i u)_l^{-h}\nabla_l^{-h}H_1[u]\Big)\bigg). 
	\end{align}
	After applying the difference quotient product rule 
	\begin{equation}\label{m}
	\nabla^h_l (uv)(x) = u^h_l(x)\nabla_l^h v(x) + v(x)\nabla_l^h u(x)
	\end{equation}
	to the integrand in the last line of \eqref{3l}, we may then estimate the last two lines of \eqref{3l} in the usual way. Namely, after applying the bound $\operatorname{tr}(F)\leq C(\Delta u + C_1)^{k-1}$, using H\"older's inequality and appealing to \eqref{60}, we see that the last two lines of \eqref{3l} are collectively bounded from below by $-CJ_h^{(q+k-1)}$. The estimate \eqref{37} then follows.
\end{proof}

\begin{proof}[Proof of the discrete Bochner identity (Lemma \ref{H'})]
	 Using the product rule \eqref{m} to first calculate $\nabla^{-h}_l(H_1[u]|\nabla u|^2)$, we see
	\begin{align}
	&\Delta_{ll}^h \big(H_1[u]|\nabla u|^2\big) = \nabla^h_l\Big(\nabla_l^{-h}\big(H_1[u]\nabla^i u \nabla_i u\big)\Big) \nonumber \\
	& \,\, = \nabla^h_l \Big((H_1[u]\nabla^i u)^{-h}_l \nabla_l^{-h}\nabla_i u \Big) + \nabla_l^h \Big(H_1[u]\nabla_i u\nabla_l^{-h}\nabla^i u\Big) + \nabla_l^h\Big(\nabla_i u(\nabla^i u)_l^{-h}\nabla_l^{-h}H_1[u]\Big). \nonumber 
	\end{align}
	On the other hand, noting that $\nabla_l^h u_l^{-h} u(x) = \nabla_l^{-h} u(x)$ and $(\nabla^{-h}_l u)^h_l(x) = \nabla^h_l u(x)$, we also have by \eqref{m} the identities
	\begin{align*}
	\nabla^h_l  \Big((H_1[u]\nabla^i u)^{-h}_l \nabla_l^{-h}\nabla_i u \Big) & = H_1\nabla^i u \nabla^h_l\nabla_l^{-h}\nabla_i u + \nabla_l^{-h}\nabla_i u \nabla_l^{-h}(H_1[u]\nabla^i u) \nonumber \\
	& = H_1\nabla^i u \nabla_i\Delta_{ll}^h u + (H_1[u])^{-h}_l\big|\nabla\nabla_l^{-h} u\big|^2 + \nabla_l^{-h}\nabla_i u \nabla^i u \nabla_l^{-h}H_1[u]
	\end{align*}
	and 
	\begin{align*}
	\nabla_l^h \Big(H_1[u]\nabla_i u\nabla_l^{-h}\nabla^i u\Big)  & = \nabla^h_l \nabla^i u \nabla_l^h\Big(H_1[u]\nabla_i u \Big) + H_1\nabla_i u \nabla_l^h\nabla_l^{-h}\nabla^i u \nonumber \\
	& = (H_1[u])^h_l\big|\nabla\nabla_l^h u\big|^2 + \nabla_l^h\nabla^i u \nabla_i u \nabla_l^h H_1[u] + H_1\nabla_i u \nabla^i\Delta_{ll}^h u. 
	\end{align*}
	Putting these three identities together, we arrive at \eqref{q}. 
\end{proof}

\subsubsection{Proof of Lemma \ref{71} in the general case}\label{D}

We now prove Lemma \ref{71} in the general case. To simplify our analysis, we will make use of the following semi-convexity property of $H\in C_{\operatorname{loc}}^{1,1}(\Omega\times\mathbb{R}\times\mathbb{R}^n;\operatorname{Sym}_n(\mathbb{R}))$: there exists a constant $C_\Sigma>0$ such that the mapping $\xi\longmapsto H(x,z,\xi) + C_\Sigma|\xi|^2 I$ is convex for all $(x,z,\xi)\in \Sigma$ (this is an immediate consequence of the $C_{\operatorname{loc}}^{1,1}$ regularity of $H$). We will make use of this property in the form
\begin{align}\label{52}
H_{ij}(x,z,\xi)  \geq H_{ij}(x,z,\zeta)+ \frac{\partial H_{ij}}{\partial \xi_a}(x,z,\zeta)(\xi-\zeta)_a - C_\Sigma\delta_{ij}|\xi-\zeta|^2
\end{align}
for all $(x,z,\xi),(x,z,\zeta)\in \Sigma$. Note that in Case 1 of Theorem \ref{8}, we may take $C_\Sigma=0$ in \eqref{52}, as $H(x,z,\xi)=H_1(x,z)|\xi|^2I$ is convex with respect to $\xi$ when $H_1\geq 0$. The inequality \eqref{52} will play a role similar to that of the discrete Bochner identity used in the previous subsection (see Lemma \ref{H'}).

\begin{proof}[Proof of Lemma \ref{71}]
We first prove Lemma \ref{71} a). It suffices to show that
\begin{align}\label{37''}
F^{ij}\Delta_{ll}^h (H[u])_{ij} \geq F^{ij}\frac{\partial H_{ij}}{\partial \xi_a}[u]\,\nabla_a\Delta_{ll}^h u + \text{error terms} \quad \forall l\in\{1,\dots,n\},
\end{align}
where the error terms satisfy
\begin{equation}\label{51}
\int_{B_{R+2\rho}}\eta(\tilde{v}^+)^{q-1}|\text{error terms}| \leq  CJ_h^{(q+k)}.
\end{equation}
To keep notation succinct, we denote $x^\pm = x\pm he_l$ and $F=(F^{ij})$ in what follows. \newline

\noindent \textbf{\textit{Step 1:}} We first prove a lower bound for $F^{ij}(x) \Delta_{ll}^h (H[u](x))_{ij}$, identifying the error terms in \eqref{37''}. Observe that by \eqref{52} and the fact that $F^{ij}$ is positive definite in $\Gamma_k^+$, we have
\begin{align}
& \frac{F^{ij}(x)}{h^2}\bigg[(H[u](x^\pm))_{ij} - H(x^\pm, u(x^\pm), \nabla u(x))_{ij}\bigg] \nonumber \\
& \, \geq \frac{F^{ij}(x)}{h^2}\frac{\partial H_{ij}}{\partial \xi_a}\big(x^\pm,u(x^\pm),\nabla u(x)\big)\big(\nabla_a u(x^\pm )-\nabla_a u(x)\big) - \frac{C_\Sigma|F|}{h^2}\big|\nabla u(x^\pm)-\nabla u(x)\big|^2 \nonumber \\
& \geq \frac{F^{ij}(x)}{h^2}\frac{\partial H_{ij}}{\partial \xi_a}[u](x)\big(\nabla_a u(x^\pm)-\nabla_a u(x)\big) - \frac{C_\Sigma|F|}{h^2}\big|\nabla u(x^\pm)-\nabla u(x)\big|^2 \nonumber \\
& \quad\, - \frac{C|F|}{|h|}\big|\nabla u(x^\pm)- \nabla u(x)\big| \quad\text{for a.e. }x\in B_{R+2\rho}, \nonumber 
\end{align} 
where to obtain the second inequality we have estimated 
\begin{equation*}
\bigg|\frac{\partial H_{ij}}{\partial \xi_a}\big(x^\pm,u(x^\pm),\nabla u(x)\big)  - \frac{\partial H_{ij}}{\partial \xi_a}[u](x)\bigg| \leq \|H\|_{C^{1,1}(\Sigma)}\big(|x^\pm-x| + |u(x^\pm)-u(x)|\big) \leq C|h|. 
\end{equation*}
Recalling the definition of $\Delta_{ll}^h (H[u](x))_{ij}$, we therefore see that for a.e. $x\in B_{R+2\rho}$,
\begin{align}\label{34}
F^{ij}(x) &\Delta_{ll}^h (H[u](x))_{ij} \nonumber \\
& \, \geq F^{ij}(x)\frac{\partial H_{ij}}{\partial \xi_a}[u](x)\nabla_a \Delta_{ll}^h u(x) \nonumber \\
& \quad + \frac{F^{ij}(x)}{h^2}\bigg(H(x^+,u(x^+),\nabla u(x))_{ij} - 2(H[u](x))_{ij}+  H(x^-,u(x^-),\nabla u(x))_{ij}\bigg) \nonumber \\
& \quad  - C_\Sigma|F||\nabla_l^h\nabla u|^2 -  C_\Sigma|F||\nabla_l^{-h}\nabla u|^2 - C|F||\nabla_l^h\nabla u| - C|F||\nabla_l^{-h}\nabla u|.
\end{align}

\noindent\textbf{\textit{Step 2:}} To prove \eqref{37'}, we need to show that the error terms in last two lines of \eqref{34} satisfy \eqref{51}. Formally, these terms behave like $|F|(|\nabla^2 u|^2 + |\nabla^2 u|)$, and so by the estimate $|F| \leq C(\Delta u + C_1)^{k-1}$, the bound \eqref{51} is then conceivable. We now give the details.

 Denote the terms on the penultimate line of \eqref{34} collectively by $E_1$, and the terms on the last line of \eqref{34} collectively by $E_2$. The error terms in $E_2$ are easier to deal with. Indeed, by the bound $|F| \leq C(\Delta u + C_1)^{k-1}$, H\"older's inequality and \eqref{60}, we have
\begin{align}\label{58}
&\int_{B_{R+2\rho}}\eta(\tilde{v}^+)^{q-1}|F||\nabla_l^{\pm h}\nabla u|^2  \leq C (J_h^{(q+k)})^\frac{q+k-2}{q+k} \bigg(\int_{B_{R+2\rho}}|\nabla_l^{\pm h}\nabla u|^{q+k}\bigg)^\frac{2}{q+k}  \leq CJ_h^{(q+k)}.
\end{align}
 In exactly the same way, one can show $\int_{B_{R+2\rho}}\eta(\tilde{v}^+)^{q-1}|F||\nabla_l^{\pm h} \nabla u|\leq C J_h^{(q+k-1)}$, and combining these estimates we obtain  $\int_{B_{R+2\rho}}\eta(\tilde{v}^+)^{q-1}|E_2| \leq  CJ_h^{(q+k)}$. 

We now treat the error terms in $E_1$. We first observe that by the fundamental theorem of calculus followed by the chain rule, we have the identities
\begin{align}
&H(x^\pm,u(x^\pm),\xi)_{ij} - H(x,u(x),\xi)_{ij} \nonumber \\
& = \int_0^1 \frac{d}{dt}H(x\pm the_l,u(x^\pm),\xi)_{ij}\,dt + \int_0^1 \frac{d}{dt}H(x,u(x\pm the_l),\xi)_{ij}\,dt \nonumber \\
& = \pm h\int_0^1 \frac{\partial H_{ij}}{\partial x^l}(x\pm the_l,u(x^\pm),\xi)\,dt \pm h\int_0^1 \frac{\partial H_{ij}}{\partial z}(x,u(x\pm the_l), \xi)\nabla_l u(x\pm the_l)\,dt, \nonumber 
\end{align}
and therefore
\begin{align}\label{31}
&H(x^+,u(x^+),\xi)_{ij} - 2H(x,u(x),\xi)_{ij} + H(x^-,u(x^-),\xi)_{ij} \nonumber \\
& = h\int_0^1\bigg(\frac{\partial H_{ij}}{\partial z}(x,u(x+the_l), \xi)\nabla_l u(x+the_l) - \frac{\partial H_{ij}}{\partial z}(x,u(x-the_l),\xi)\nabla_l u(x-the_l) \bigg)\,dt \nonumber \\
& \quad + h\int_0^1\bigg(\frac{\partial H_{ij}}{\partial x^l}(x+the_l,u(x^+),\xi) - \frac{\partial H_{ij}}{\partial x^l}(x-the_l,u(x^-),\xi) \bigg)\,dt.
\end{align}  
Now, by the $C_{\operatorname{loc}}^{1,1}$ regularity of $H$ and the Lipschitz regularity of the mapping $(x,z,p)\mapsto\frac{\partial H_{ij}}{\partial z}(x,z,\xi)p_l$ for fixed $\xi$ and each $l\in\{1,\dots,n\}$, we can estimate the last line of \eqref{31} from above by $Ch^2$ and the middle line of \eqref{31} from above by 
\begin{equation*}
 Ch^2 + Ch^2\int_0^1 \frac{1}{t|h|}\Big|\nabla_l u(x+the_l) - \nabla_l u(x-the_l)\Big|\,dt.
\end{equation*}
Applying these estimates in \eqref{31} and taking $\xi=\nabla u(x)$, we therefore see that 
\begin{align}\label{31'}
|E_1| & \leq C|F| +   C|F|\int_0^1|\nabla_l^{th}\nabla_l u(x)|\,dt  + C|F|\int_0^1 |\nabla_l^{-th}\nabla_l u(x)|\,dt. 
\end{align}

Using \eqref{31'}, one readily obtains the estimate $\int_{B_{R+2\rho}}\eta(\tilde{v}^+)^{q-1}|E_1| \leq  CJ_h^{(q+k-1)}$, applying the same line of argument as seen above for $E_2$. For example, by Fubini's theorem and Young's inequality, we have
\begin{align}
\int_{B_{R+2\rho}}\eta(\tilde{v}^+)^{q-1}|F|\bigg(\int_0^1&|\nabla^{\pm th}_l\nabla_l u(x)|\,dt\bigg)\,dx  = \int_0^1\int_{B_{R+2\rho}}\eta(\tilde{v}^+)^{q-1}|F||\nabla_l^{\pm th}\nabla_l u(x)|\,dx\,dt \nonumber \\
&   \leq  C J_h^{(q+k-1)} + C \int_0^1\int_{B_{R+2\rho}}|\nabla^{\pm th}_l \nabla_l u|^{q+k-1}\,dx\,dt \stackrel{\eqref{60}}{\leq} C J_h^{(q+k-1)}. \nonumber 
\end{align}
This completes the proof of Lemma \ref{71} a).  \newline

\noindent\textbf{\textit{Step 3:}} It remains to prove Lemma \ref{71} b) (see Section \ref{3n} for an alternative proof which is independent of calculations in Steps 1 and 2 above). Note that in this case, we may take $C_\Sigma=0$ in \eqref{34} and so the error terms on the last two lines of \eqref{34} formally behave like $|F||\nabla^2 u|$. By the same argument as in Step 2, the error terms $E_1$ and $E_2$ considered in Step 2 therefore satisfy $\int_{B_{R+2\rho}}\eta(\tilde{v}^+)^{q-1}|E_i| \leq CJ_h^{(q+k-1)}$, and the conclusion follows. 
 \end{proof}

\section{Proof of main results}\label{82}

In this section we use Corollaries \ref{J}, \ref{M} and \ref{N} to prove Theorems \ref{8} and \ref{56}, as outlined at the end of Section \ref{79}. We will first give a detailed proof of Case 1 of Theorem \ref{8} when $f=f(x,z)$ in Section \ref{AH}, and then indicate the necessary adjustments for remaining cases, still when $f=f(x,z)$, in Section \ref{AH'}. In Section \ref{AT}, we extend these results to the case $f=f(x,z,\xi)$, completing the proofs of Theorems \ref{8} and \ref{56}.

\subsection{Proof of Case 1 of Theorem \ref{8} when $f=f(x,z)$}\label{AH}
In this case, we recall that by Corollary \ref{J} we have the estimate
\begin{equation}\label{r'}
\begin{split}
\int_{B_{R+\rho}} {f}^k \frac{\big|\nabla\big((\tilde{v}^+)^{q/2}\big)\big|^2}{\Delta u - \operatorname{tr}(H)}  \leq \frac{Cq}{\rho^2}J_h^{(q+k-1)},
\end{split}
\end{equation}
where $u\in W_{\operatorname{loc}}^{2,q+k-1}(\Omega)\cap W_{\operatorname{loc}}^{1,\infty}(\Omega)$ ($q>1$). Let $\theta\in(0,1)$ be such that $\frac{2-\theta}{\theta}\leq q+k-1$ (we will eventually take $\theta = \frac{4}{kn+2}$). Also denote by $(2-\theta)^* \defeq n(2-\theta)/(n-2+\theta)$ the Sobolev conjugate of $2-\theta$. We first obtain from \eqref{r'} the following:

\begin{lem}\label{AL}
	Suppose $f\in C_{\operatorname{loc}}^{1,1}(\Omega\times\mathbb{R})$ is positive, $H = H_1(x,z)|\xi|^2 I$ with $H_1\in C_{\operatorname{loc}}^{1,1}(\Omega\times\mathbb{R})$ and $H_1\geq 0$, and that $u\in W_{\operatorname{loc}}^{2,q+k-1}(\Omega)\cap W_{\operatorname{loc}}^{1,\infty}(\Omega)$ ($q>1$) is a solution to \eqref{7''}. Then 
	\begin{equation}\label{-7}
	\bigg(\int_{B_{R+\rho}}(\Delta u + C_1)^{\frac{q(2-\theta)^*}{2}}\bigg)^{\frac{2}{(2-\theta)^*}} \leq \frac{Cq}{\rho^2}\bigg(\int_{B_{R+3\rho}}(\Delta u + C_1)^{\frac{2-\theta}{\theta}}\bigg)^{\frac{\theta}{2-\theta}}\int_{B_{R+3\rho}}(\Delta u + C_1)^{q+k-1}.
	\end{equation}
	\end{lem}
	\begin{proof}
	The estimate \eqref{-7} will follow immediately once we establish the estimate
		\begin{align}\label{AK}
			\bigg( \int_{B_{R+\rho}}(\tilde{v}^+)^{\frac{q(2-\theta)^*}{2}}\bigg)^{\frac{2}{(2-\theta)^*}}  \leq \frac{Cq}{\rho^2}\big| J_h^{(\frac{2-\theta}{\theta})}\big|^{\frac{\theta}{2-\theta}}  J_h^{(q+k-1)},
		\end{align}
		since we can then apply Fatou's lemma and the fact that $\tilde{v}^+\rightarrow \Delta u + C_1$ a.e. as $h\rightarrow 0$ to the term on the LHS of \eqref{AK}, and Lemma \ref{3g} to the terms on the RHS of \eqref{AK}. 
		
		Keeping in mind the lower bound $\inf_{B_{2R}} f >\frac{1}{C} >0$, we first observe that by H\"older's inequality and \eqref{r'}, we have
\begin{align}
\bigg(\int_{B_{R+\rho}}\big|\nabla\big((\tilde{v}^+)^{q/2}\big)\big|^{2-\theta}\bigg)^{\frac{2}{2-\theta}} & \leq \,\, C\bigg(\int_{B_{R+\rho}}(\Delta u - \operatorname{tr}(H))^{\frac{2-\theta}{\theta}}\bigg)^{\frac{\theta}{2-\theta}}\int_{B_{R+\rho}}{f}^k\frac{\big|\nabla\big((\tilde{v}^+)^{q/2}\big)\big|^2}{\Delta u - \operatorname{tr}(H)} \nonumber  \\
& \leftstackrel{\eqref{r'}}{\leq}\frac{Cq}{\rho^2}\big| J_h^{(\frac{2-\theta}{\theta})}\big|^{\frac{\theta}{2-\theta}}  J_h^{(q+k-1)}. \label{13'}
\end{align}
On the other hand, since $\frac{q(2-\theta)}{2}\leq q+k-1$,  H\"older's inequality gives
\begin{align}\label{13''}
\bigg(\int_{B_{R+\rho}}(\tilde{v}^+&)^{\frac{q(2-\theta)}{2}}\bigg)^{\frac{2}{2-\theta}} \leq \bigg(\int_{B_{R+\rho}}(\tilde{v}^+)^{\frac{2-\theta}{\theta}}\bigg)^{\frac{\theta}{2-\theta}}\int_{B_{R+\rho}}(\tilde{v}^+)^{q-1} \leq \big| J_h^{(\frac{2-\theta}{\theta})}\big|^{\frac{\theta}{2-\theta}}  J_h^{(q+k-1)}. 
\end{align}
Applying the Sobolev inequality to $(\tilde{v}^+)^{q/2}\in W^{1,2-\theta}$, and appealing to \eqref{13'} and \eqref{13''}, we arrive at \eqref{AK}. 
\end{proof}

The inequality \eqref{-7} is of reverse H\"older-type if $\theta$ satisfies
\begin{equation*}
\frac{2-\theta}{\theta}<q+k-1< \frac{q(2-\theta)^*}{2}.
\end{equation*} 
For example, if we fix $\theta = \frac{4}{kn+2}$ and finally impose the assumption $q+k-1>\frac{kn}{2}$, we see that $(2-\theta)/\theta = kn/2 <q+k-1$ and
\begin{align}
\frac{q(2-\theta)^*}{2} - (q + k - 1 )	 > \bigg(\frac{kn}{2}-k+1\bigg)\bigg(\frac{kn}{2+kn-2k}-1\bigg) - k + 1  = 0. \nonumber 
\end{align}
In what follows, we denote
\begin{equation*}
\beta \defeq \frac{(2-\theta)^*}{2} = \frac{kn}{kn+2-2k}>1. 
\end{equation*}

\begin{proof}[Proof of Case 1 of Theorem \ref{8} when $f=f(x,z)$]
With $\theta = \frac{4}{kn+2}$, we obtain from \eqref{-7} the estimate
	\begin{align}\label{41}
	\bigg(\int_{B_{R+\rho}}(\Delta u + C_1)^{\beta q}\bigg)^{1/\beta} \leq \frac{Cq}{\rho^2}\int_{B_{R+3\rho}}(\Delta u + C_1)^{q+k-1}
	\end{align}
	for all $q>\frac{kn}{2}-k+1$ and $\rho\in(0,\frac{R}{3}]$. The constant $C$ in \eqref{41} and below now depends on $\int_{B_{R+3\rho}}(\Delta u + C_1)^{kn/2}$, which is finite due to our hypotheses.
	
	We now carry out the Moser iteration argument. Let $p>\frac{kn}{2}$ be as in the statement of Theorem \ref{8}, and define a sequence $q_j$ inductively by
	\begin{equation*}
	q_0 = p-k+1, \quad q_j = \beta q_{j-1} - k+1 \mathrm{~for~}j\geq 1. 
	\end{equation*}
	Then $q_j = \beta q_{j-1}-(k-1)  = \beta^j q_0 - (k-1)(\beta^{j-1}+\cdots + \beta + 1)$, which implies
	\begin{equation}\label{BC}
	\frac{q_j}{\beta^j} = q_0-(k-1)\bigg(\frac{1-\beta^{-j}}{\beta -1}\bigg) \stackrel{j\rightarrow \infty}{\longrightarrow} q_0 - \frac{k-1}{\beta-1} >0. 
	\end{equation}
	Note that the limit in \eqref{BC} is positive by definition of $\beta$ and the fact that $q_0>\frac{kn}{2} -k+1$. In particular, $q_j\rightarrow\infty$ as $j\rightarrow \infty$. 
	
Applying \eqref{41} iteratively with $q=q_j$ and $\rho=3^{-j-1}R$, we have for each $j\geq 0$
	\begin{align}
	\bigg(\int_{B_{(1+3^{-j-1})R}}(\Delta u + C_1)^{\beta q_j}\bigg)^{\beta^{-j-1}}&  \leq  \bigg(9^{j}Cq_j \int_{B_{(1+3^{-j})R}}(\Delta u + C_1)^{\beta q_{j-1}}\bigg)^{\beta^{-j}} \nonumber \\
	& \leftstackrel{\eqref{BC}}{\leq} \prod_{i=0}^j \big((9\beta)^i C\big)^{\beta^{-i}} \int_{B_{2R}}(\Delta u + C_1)^{p} \nonumber \\
	& \leq (9\beta)^{\sum_{i=0}^\infty i\beta^{-i}} C^{\sum_{i=0}^\infty \beta^{-i}} \int_{B_{2R}}(\Delta u + C_1)^{p}. \nonumber 
	\end{align}
	 Letting $j\rightarrow\infty$ and appealing once again to \eqref{BC}, we arrive at
\begin{equation*}
\|\Delta u + C_1\|_{L^\infty(B_R)} \leq C\bigg(\int_{B_{2R}}(\Delta u+C_1)^p\bigg)^{\big(q_0-\frac{k-1}{\beta-1}\big)^{-1}},
\end{equation*}
which implies the desired bound on $\|\nabla^2 u\|_{L^\infty(B_R)}$ by the choice of $C_1$. 
\end{proof} 

\subsection{Proof of Case 2 of Theorem \ref{8} and Theorem \ref{56} when $f=f(x,z)$}\label{AH'}

In these cases, we recall that by Corollaries \ref{M} and \ref{N} we have the estimate
\begin{equation}\label{s'}
\begin{split}
\int_{B_{R+\rho}} {f}^k \frac{\big|\nabla\big((\tilde{v}^+)^{q/2}\big)\big|^2}{\Delta u - \operatorname{tr}(H)}  \leq \frac{Cq}{\rho^2}J_h^{(q+k)},
\end{split}
\end{equation}
 where $u\in W_{\operatorname{loc}}^{2,q+k}(\Omega)\cap W_{\operatorname{loc}}^{1,\infty}(\Omega)$ ($q>1$).

\begin{proof}[Proof of Case 2 of Theorem \ref{8} and Theorem \ref{56} when $f=f(x,z)$]
We let $\theta\in (0,1)$ be such that $\frac{2-\theta}{\theta}\leq q+k$. Following the same arguments as in Section \ref{AH}, one readily obtains the following counterpart to the estimate \eqref{-7}:
\begin{equation}\label{-7'}
\bigg(\int_{B_{R+\rho}}(\Delta u + C_1)^{\frac{q(2-\theta)^*}{2}}\bigg)^{\frac{2}{(2-\theta)^*}} \leq \frac{Cq}{\rho^2}\bigg(\int_{B_{R+3\rho}}(\Delta u + C_1)^{\frac{2-\theta}{\theta}}\bigg)^{\frac{\theta}{2-\theta}}\int_{B_{R+3\rho}}(\Delta u + C_1)^{q+k}.
\end{equation}
Taking $\theta = \frac{4}{(k+1)n+2}$ and imposing $q+k>\frac{(k+1)n}{2}$, we see
\begin{equation*}
\frac{2-\theta}{\theta} = \frac{(k+1)n}{2}< q+k < \frac{q(2-\theta)^*}{2}. 
\end{equation*}
We thus obtain from \eqref{-7'} the estimate
\begin{align}
\bigg(\int_{B_{R+\rho}}(\Delta u + C_1)^{\beta q}\bigg)^{1/\beta} \leq \frac{Cq}{\rho^2}\int_{B_{R+3\rho}}(\Delta u + C_1)^{q+k}, \nonumber 
\end{align}
where
\begin{equation*}
\beta \defeq \frac{(k+1)n}{(k+1)n + 2 - 2(k+1)}> 1
\end{equation*}
and $C$ now depends on $\int_{B_{R+3\rho}}(\Delta u + C_1)^{(k+1)n/2}$. The Moser iteration argument then follows through as before, using $p>\frac{(k+1)n}{2}$ and defining $q_j$ inductively by $q_0 = p - k$ and $q_j = \beta q_{j-1} -k \text{ for }j\geq 1$.
\end{proof}

\subsection{Proof of Theorems \ref{8} and \ref{56} for $f=f(x,z,\xi)$}\label{AT}

In this section we explain how the preceding arguments may be adjusted to treat the general case $f=f(x,z,\xi)$, thus completing the proofs of Theorems \ref{8} and \ref{56}:

\begin{proof}[Proof of Theorems \ref{8} and \ref{56}]
The arguments up until \eqref{AF} remain valid for $f=f(x,z,\xi)$, but the last term in \eqref{12'} can no longer be estimated as in \eqref{AF}. Consequently, under otherwise the same hypotheses, the conclusion of Lemma \ref{P} now reads
	\begin{equation*}
\operatorname{(I_1)}_h + \operatorname{(I_2)}_h + \operatorname{(I_3)}_h + \operatorname{(I_4)}_h \leq C\rho^{-2}J_h^{(q+k-1)},
\end{equation*}
where $\operatorname{(I_1)}_h$, $\operatorname{(I_2)}_h$ and $\operatorname{(I_3)}_h$ are as before and 
\begin{equation*}
\operatorname{(I_4)}_h \defeq \sum_l \int_{B_{R+2\rho}}k\eta(\tilde{v}^+)^{q-1}f^{k-1}\Delta_{ll}^h f[u].
\end{equation*}

The estimates for $\operatorname{(I_1)}_h$, $\operatorname{(I_2)}_h$ and $\operatorname{(I_3)}_h$ are unchanged (see Lemmas \ref{21'}, \ref{83}, \ref{83'}, \ref{84} and \ref{71}), since they do not involve differentiating $f$. The integrand of $\operatorname{(I_4)}_h$ was previously a lower order term, but is now formally of third order in $u$. However, this can be treated using some of the ideas already seen in the proof of Lemma \ref{71}. Indeed, by the same argument leading to \eqref{34}, we have for each $l\in\{1,\dots,n\}$ and a.e. $x\in B_{R+2\rho}$ the estimate
\begin{align}\label{34'}
\Delta_{ll}^h f[u](x) & \geq\frac{\partial f}{\partial \xi_a}[u](x)\nabla_a \Delta_{ll}^h u(x) - C_\Sigma|\nabla_l^h\nabla u|^2 -  C_\Sigma|\nabla_l^{-h}\nabla u|^2 - C|\nabla_l^h\nabla u| - C|\nabla_l^{-h}\nabla u| \nonumber \\
& \quad + \frac{1}{h^2}\bigg(f(x^+,u(x^+),\nabla u(x)) - 2f[u](x)  +  f(x^-,u(x^-),\nabla u(x))\bigg).
\end{align}
As before, the constant $C_\Sigma>0$ is such that the mapping $\xi \longmapsto f(x,z,\xi)+C_\Sigma|\xi|^2$ is convex for all $(x,z,\xi)\in \Sigma$. Denoting all but the first term on the RHS of \eqref{34'} as error terms, it follows from \eqref{34'} that
\begin{align}\label{62}
\operatorname{(I_4)}_h \geq  \int_{B_{R+2\rho}}k\eta(\tilde{v}^+)^{q-1}f^{k-1}\frac{\partial f}{\partial \xi_a}[u]\nabla_a\tilde{v} - \int_{B_{R+2\rho}}k\eta(\tilde{v}^+)^{q-1}f^{k-1}|\text{error terms}|.
\end{align}
Now, in the same way that we dealt with the error terms in Step 2 of the proof of Lemma \ref{71}, one readily obtains $\int_{B_{R+2\rho}}k\eta(\tilde{v}^+)^{q-1}f^{k-1}|\text{error terms}| \leq CJ_h^{(q+1)}$. For the first integral on the RHS of \eqref{62}, we integrate by parts and apply H\"older's inequality to obtain
\begin{align*}
\bigg|\int_{B_{R+2\rho}}k\eta(\tilde{v}^+)^{q-1}f^{k-1}\frac{\partial f}{\partial \xi_a}[u]\nabla_a\tilde{v} \bigg| \leq C\rho^{-1}J_h^{(q+1)}.
\end{align*}
Returning to \eqref{62}, we therefore obtain $\operatorname{(I_4)}_h \geq - C\rho^{-1}J_h^{(q+1)}$. As a consequence, the estimates \eqref{r'} and \eqref{s'} hold, and the arguments of Section \ref{82} therefore apply without any changes.
\end{proof}

\section{The case $k\geq 3$ for general $H$}\label{BA}

In this final section we consider a minor extension of Theorem \ref{56}. Recall that our proof of Theorems \ref{8} and \ref{56} exploited a cancellation phenomenon between higher order terms arising from $\operatorname{(I_2)}_h$ and $\operatorname{(I_3)}_h$, where the divergence structure of $F^{ij}$ played a role in estimating $\operatorname{(I_2)}_h$. When $3\leq k\leq n$ and $H$ is not necessarily a multiple of the identity, the divergence structure given in \eqref{15} is more involved and the resulting arguments fall outside the scope of the present paper. That said, if one assumes higher integrability on $\nabla^2 u$ from the outset, the terms $\operatorname{(I_2)}_h$ and $\operatorname{(I_3)}_h$ may be estimated by using Cauchy's inequality and absorbing the resulting negative higher order terms into the positive term $\operatorname{(I_1)}_h$. This avoids the need to prove any cancellation between $\operatorname{(I_2)}_h$ and $\operatorname{(I_3)}_h$. We establish:

\begin{thm}\label{BB}
	Let $\Omega$ be a domain in $\mathbb{R}^n$ ($n\geq 3$), $f=f(x,z,\xi)\in C_{\operatorname{loc}}^{1,1}(\Omega\times\mathbb{R}\times\mathbb{R}^n)$ a positive function and $H\in C_{\operatorname{loc}}^{1,1}(\Omega\times\mathbb{R}\times\mathbb{R}^n\,;\,\operatorname{Sym}_n(\mathbb{R}))$. Suppose $3\leq k\leq n$, $p>kn$ and $u\in W_{\operatorname{loc}}^{2,p}(\Omega)$ is a solution to \eqref{7}. Then $u\in C^{1,1}_{\operatorname{loc}}(\Omega)$, and for any concentric balls $B_{R}\subset B_{2R} \Subset \Omega$ we have
	\begin{equation*}
	\|\nabla^2 u \|_{L^\infty(B_R)}\leq C,
	\end{equation*}
	where $C$ is a constant depending only on $n,p,R, f, H$ and an upper bound for $\| u\|_{W^{2,p}(B_{2R})}$.
\end{thm} 

\begin{proof}
	Following the proof of Theorem \ref{56} in Section \ref{AT} but leaving the terms $\operatorname{(I_2)}_h$ and $\operatorname{(I_3)}_h$ untreated, we have for $u\in W_{\operatorname{loc}}^{2,q+k-1}(\Omega)\cap W_{\operatorname{loc}}^{1,\infty}(\Omega)$ ($q>1$)
	\begin{align}\label{AI}
	\frac{q-1}{Cq^2}\int_{B_{R+2\rho}}\eta \frac{\big|\nabla\big((\tilde{v}^+)^{q/2}\big)\big|^2}{\operatorname{tr}(A_H)} + \operatorname{(I_2)}_h + \operatorname{(I_3)}_h  \leq C\rho^{-2}J_h^{(q+k-1)}. 
	\end{align}
	We now suppose further that $\nabla^2 u \in L_{\operatorname{loc}}^{q+2k-1}(\Omega)$ ($q>1$). By Cauchy's inequality and the bound $|\operatorname{div}F[u]| \leq C(\Delta u + C_1)^{k-1}$ in equation \eqref{3h}, we see that for all $\delta>0$
	\begin{align}\label{AA}
	\operatorname{(I_2)}_h & = \frac{2}{q}\int_{B_{R+2\rho}}\eta(\tilde{v}^+)^{q/2}\nabla_i F[u]^{ij}\nabla_j(\tilde{v}^+)^{q/2} \nonumber \\
	& \geq -\frac{\delta(q-1)}{q^2}\int_{B_{R+2\rho}}\eta  \frac{\big|\nabla\big((\tilde{v}^+)^{q/2}\big)\big|^2}{\operatorname{tr}(A_H)} - \frac{1}{\delta(q-1)}\int_{B_{R+2\rho}}\eta (\tilde{v}^+)^q\operatorname{tr}(A_H)\big|\operatorname{div}F[u]\big|^2 \nonumber \\
	& \geq -\frac{\delta(q-1)}{q^2}\int_{B_{R+2\rho}}\eta  \frac{\big|\nabla\big((\tilde{v}^+)^{q/2}\big)\big|^2}{\operatorname{tr}(A_H)} - \frac{C}{\delta(q-1)}J_h^{(q+2k-1)}.
	\end{align}
	By similar reasoning, it also holds that
	\begin{align}\label{AA'}
	\operatorname{(I_3)}_h &  \stackrel{\eqref{37'}}{\geq} \int_{B_{R+2\rho}}\eta (\tilde{v}^+)^{q-1}F^{ij}\frac{\partial H_{ij}}{\partial \xi_a}[u]\,\nabla_a \tilde{v} -C J_h^{(q+k)} \nonumber \\
	&\,\,\, \geq -\frac{\delta(q-1)}{q^2}\int_{B_{R+2\rho}}\eta  \frac{\big|\nabla\big((\tilde{v}^+)^{q/2}\big)\big|^2}{\operatorname{tr}(A_H)} - \frac{C}{\delta(q-1)}J_h^{(q+2k-1)}.
	\end{align}
Taking $\delta$ sufficiently small in \eqref{AA} and \eqref{AA'}, and then substituting these estimates into \eqref{AI}, we obtain
	\begin{align}\label{AJ'}
	\frac{q-1}{q^2}&\int_{B_{R+2\rho}}\eta \frac{\big|\nabla\big((\tilde{v}^+)^{q/2}\big)\big|^2}{\operatorname{tr}(A_H)}  \leq C\rho^{-2} J_h^{(q+2k-1)}. 
	\end{align}
	The argument then proceeds as in Section \ref{AH}: we let $\theta\in(0,1)$ be such that $\frac{2-\theta}{\theta}\leq q+2k-1$ and obtain from \eqref{AJ'} the estimate
	\begin{equation}\label{-7''}
	\bigg(\int_{B_{R+\rho}}(\Delta u + C_1)^{\frac{q(2-\theta)^*}{2}}\bigg)^{\frac{2}{(2-\theta)^*}} \leq \frac{Cq}{\rho^2}\bigg(\int_{B_{R+3\rho}}(\Delta u + C_1)^{\frac{2-\theta}{\theta}}\bigg)^{\frac{\theta}{2-\theta}}\int_{B_{R+3\rho}}(\Delta u + C_1)^{q+2k-1}.
	\end{equation}
	Taking $\theta = \frac{2}{kn+1}$ and imposing $q+2k-1>kn$, we see that $
\frac{2-\theta}{\theta} = kn<q+2k-1 < \frac{q(2-\theta)^*}{2}$, and we therefore obtain from \eqref{-7''} the estimate
	 \begin{align}
	\bigg(\int_{B_{R+\rho}}(\Delta u + C_1)^{\beta q}\bigg)^{1/\beta} \leq \frac{Cq}{\rho^2}\int_{B_{R+3\rho}}(\Delta u + C_1)^{q+2k-1},\nonumber 
	\end{align}
	where $\beta \defeq kn/(kn+1-2k) > 1$ and $C$ now depends on $\int_{B_{R+3\rho}}(\Delta u + C_1)^{kn}$. The Moser iteration argument then goes through as before, giving the desired conclusion. 
\end{proof}

\appendix

\section{A remark on the regularity of solutions to the $\sigma_2$-Yamabe equation obtained by vanishing viscosity}\label{APPA}

Let $(M^4,g_0)$ be a 4-manifold with scalar curvature $R_0>0$ and Schouten tensor $A_0$ satisfying $\int_{M^4}\sigma_2(A_0)\,dv_0>0$.  In \cite{CGY02a}, the existence of smooth solutions $g_{w_\delta}=e^{2w_{\delta}}g_0$ with positive scalar curvature to the fourth order equation
\begin{equation}\tag{A.1}\label{AP2}
\sigma_2(A_{g_{w_\delta}}) = \frac{\delta}{4}\Delta_{g_{w_\delta}} R_{g_{w_\delta}} - 2\gamma_1|\eta|_{g_{w_\delta}}^2
\end{equation}
is established for each $\delta\in(0,1]$, where $\eta$ is any fixed non-vanishing $(0,2)$-tensor and $\gamma_1<0$ is the conformal invariant obtained by integrating both sides of \eqref{AP2}. Moreover, solutions are shown to satisfy the uniform estimates
\begin{equation}\tag{A.2}\label{AP1}
\|w_\delta\|_{W^{2,s}(M^4,g_0)}\leq C\quad\mathrm{for~all~}\delta\in(0,1], \quad 1\leq s<5,
\end{equation} 
where the constant $C=C(s)$ is independent of $\delta$. A heat flow argument is then applied to obtain a conformal metric $g$ with $\lambda(A_g)\in\Gamma_2^+$. In this appendix, we show that in the case that $(M^4,g_0)$ is locally conformally flat, we may take the limit $\delta\rightarrow 0$ more directly in \eqref{AP2} to obtain the desired conformal metric with $\lambda(A_g)\in\Gamma_2^+$. More precisely, using Theorem \ref{3} and a result of \cite{LN20}, we show that, along a subsequence, the solutions $w_\delta$ converge weakly to a smooth solution of the equation $\sigma_2(A_{g_{w_\delta}})=-2\gamma_1|\eta|_{g_{w_\delta}}^2>0$. 

To this end, fix $4<s<5$. We first observe that by \eqref{AP1}, we can find a sequence $\delta_i\rightarrow 0$ for which $w_i\defeq w_{\delta_i}$ converges weakly in $W^{2,s}(M^4,g_0)$, say to $w\in W^{2,s}(M^4,g_0)$. By the Morrey embedding $W^{2,s}(M^4,g_0)\hookrightarrow C^{1,1-\frac{4}{s}}(M^4,g_0)$, we may assume $w_i\rightarrow w$ in $C^{1,\alpha}(M^4,g_0)$ for some $\alpha>0$. It then follows from \cite[Proposition 5.3]{LN20} and the estimate \eqref{AP1} that for all $\phi\in C^0(M^4)$, we have
\begin{equation}\tag{A.3}\label{AP3}
\lim_{i\rightarrow \infty} \int_{M^4}\sigma_2(A_{g_{w_i}})\phi\,dv_0= \int_{M^4}\sigma_2(A_{g_{w}})\phi\,dv_0.
\end{equation}
 Substituting the equation \eqref{AP2} into \eqref{AP3} and integrating by parts, we therefore see that
\begin{align}
\int_{M^4} \sigma_2(A_{g_w})\phi\,dv_0 & = \lim_{i\rightarrow\infty}\int_{M^4} \bigg(\frac{\delta_i}{4}R_{g_{w_i}}\Delta_{g_{w_i}}\phi-2\gamma_1|\eta|_{g_{w_i}}^2\phi\bigg)\,dv_0 = -\int_{M^4}2\gamma_1|\eta|_{g_w}^2\phi\,dv_0 \nonumber 
\end{align}
for all $\phi\in C^2(M^4,g_0)$. It follows that $w\in W^{2,s}(M^4,g_0)$ solves \begin{equation}\tag{A.4}\label{AP4}
\sigma_2(A_{g_w}) = -2\gamma_1|\eta|_{g_w}^2>0\quad\mathrm{a.e.~in~}M^4.
\end{equation}
Moreover, as $R_{g_{w_i}}>0$ for each $i$, it follows that $R_{g_w}\geq 0$, and by \eqref{AP4} we therefore have $R_{g_w}>0$ a.e. If $(M^4,g_0)$ is locally conformally flat, we therefore obtain from Theorem \ref{3} that $u\defeq e^{-w}\in C^{1,1}(M^4,g_0)$, and consequently \eqref{AP4} is uniformly elliptic at $w$. 

At this point, we apply the Evans-Krylov theorem to obtain $u\in C^{2,\alpha}(M^4,g_0)$.  Indeed, by the proof of \cite[Theorem 6.6]{CC95}, it suffices to observe that, by Lemmas \ref{CC} and \ref{H'},  $v=\sum_l \Delta_{ll}^h u$ is a subsolution to a uniformly elliptic linear equation, namely
\begin{equation*}
F^{ij}\nabla_i\nabla_j v + B^iD_iv \geq C,
\end{equation*}
where $F^{ij}$ is uniformly elliptic and $F^{ij}$, $B^i$ and $C$ are bounded.  Furthermore, since $f(x,u) \defeq -2\gamma_1|\eta(x)|_{u^{-2}g_0}^2 = -2\gamma_1 u^4|\eta(x)|_{g_0}^2$ is smooth, standard elliptic regularity ensures that $u$ (and hence $w$) belongs to $C^\infty(M^4,g_0)$.

\section{Proof of Lemma \ref{3g}}\label{APPB}

	The proof is a standard argument using Taylor's theorem. We claim that 
	\begin{align}\label{1}
	\|v_h-\Delta u\|_{L^s(\Omega')} &\leq \sum_{l=1}^n\int_0^1\Big\|\nabla_l\nabla_lu(x+the_l)-\nabla_l\nabla_l u(x)\Big\|_{L^s(\Omega')} \,dt \nonumber \\
	& \quad + \sum_{l=1}^n\int_0^1\Big\|\nabla_l\nabla_lu(x-the_l)-\nabla_l\nabla_l u(x)\Big\|_{L^s(\Omega')} \,dt \tag{B.1}
	\end{align}
	for all $u\in W^{2,s}(\Omega)$ and $\Omega'\Subset\Omega$ satisfying $|h|<\operatorname{dist}(\Omega',\partial\Omega)$, from which the conclusion follows by the continuity of the translation operator in $L^s(\Omega)$. By density it suffices to prove \eqref{1} for $u\in C^2(\Omega)$. Let $\Omega'$ be as above. Then for each $x\in \Omega'$ and $l\in\{1,\dots,n\}$, we have by Taylor's theorem
	\begin{align}
	u(x\pm he_l) = u(x) \pm h\nabla_l u(x) + h^2\int_0^1 (1-t)\nabla_l\nabla_l u(x\pm the_l)\,dt,\nonumber
	\end{align}
	and thus
	\begin{align}\label{2}
	v_h(x) - \Delta u(x) & = \sum_{l=1}^n \int_0^1 (1-t)\bigg(\nabla_l\nabla_l u(x+the_l) - \nabla_l\nabla_l u(x)\bigg)\,dt \nonumber \\
	& \quad + \sum_{l=1}^n\int_0^1(1-t)\bigg(\nabla_l\nabla_l u(x-the_l) - \nabla_l\nabla_l u(x)\bigg)\,dt. \tag{B.2} 
	\end{align}
	Let $s'$ be such that $\frac{1}{s}+\frac{1}{s'}=1$. It follows from \eqref{2} and H\"older's inequality that for all $g\in L^{s'}(\Omega')$ satisfying $\|g\|_{L^{s'}(\Omega')}\leq 1$, we have
	
	\begin{align}\label{90}
	\int_{\Omega'}\big(v_h(x)-\Delta u(x)\big)g(x)\,dx & = \sum_{l=1}^n\int_0^1(1-t)\int_{\Omega'}\bigg(\nabla_l\nabla_l u(x+the_l) - \nabla_l\nabla_l u(x)\bigg)g(x)\,dx\,dt \nonumber \\
	& \quad + \sum_{l=1}^n\int_0^1(1-t)\int_{\Omega'}\bigg(\nabla_l\nabla_l u(x-the_l) - \nabla_l\nabla_l u(x)\bigg)g(x)\,dx\,dt \nonumber \\
& \leq  \sum_{l=1}^n\int_0^1\Big\|\nabla_l\nabla_lu(x+the_l)-\nabla_l\nabla_l u(x)\Big\|_{L^s(\Omega')} \,dt\nonumber \\
	& \quad + \sum_{l=1}^n\int_0^1\Big\|\nabla_l\nabla_lu(x-the_l)-\nabla_l\nabla_l u(x)\Big\|_{L^s(\Omega')} \,dt.\tag{B.3}
	\end{align}
Taking the supremum over such $g$ in \eqref{90}, we obtain \eqref{1}. 

\bibliographystyle{siam}

\end{document}